\documentclass[12pt]{amsart}
\usepackage[margin=1.2in]{geometry}
\usepackage{graphicx,latexsym}
\usepackage{amsfonts, amssymb, amsmath, amsthm, bm, verbatim, mathrsfs}
\usepackage{tikz}
\usetikzlibrary{matrix,arrows}
\allowdisplaybreaks
\usepackage{todonotes}

\newcommand{\N}{\mathbb{N}}
\newcommand{\Z}{\mathbb{Z}}
\newcommand{\R}{\mathbb{R}}
\newcommand{\C}{\mathbb{C}}

\newcommand{\SSS}{\mathcal{S}}

\newcommand{\DD}{\mathcal D}

\newcommand{\dx}{{\rm d}x }

\newcommand{\dt}{{\rm d}t }
\newcommand{\dxi}{{\rm d}\xi }

\newcommand{\loc}{\operatorname{loc}}
\newcommand{\beq}{\begin{eqnarray}}
\newcommand{\eeq}{\end{eqnarray}}

\newcommand{\beqs}{\begin{eqnarray*}}
\newcommand{\eeqs}{\end{eqnarray*}}

\newtheorem{theorem}{Theorem}[section]
\newtheorem{proposition}[theorem]{Proposition}
\newtheorem{lemma}[theorem]{Lemma}
\newtheorem{corollary}[theorem]{Corollary}

\theoremstyle{definition}
\newtheorem{definition}[theorem]{Definition}

\newtheorem{problem}[theorem]{Problem}
\newtheorem{examples}[theorem]{Examples}

\theoremstyle{remark}
\newtheorem{remark}[theorem]{Remark}

\numberwithin{equation}{section}

\usetikzlibrary{arrows,chains,matrix,positioning,scopes}
\makeatletter
\tikzset{join/.code=\tikzset{after node path={%
\ifx\tikzchainprevious\pgfutil@empty\else(\tikzchainprevious)%
edge[every join]#1(\tikzchaincurrent)\fi}}}
\makeatother
\tikzset{>=stealth',every on chain/.append style={join},
         every join/.style={->}}
\tikzstyle{labeled}=[execute at begin node=$\scriptstyle,
   execute at end node=$]

\begin{document}
\title[Gabor frame characterisations of modulation spaces]{Gabor frame characterisations of generalised modulation spaces}

\author[A. Debrouwere]{Andreas Debrouwere}
\thanks{A. Debrouwere was supported by  FWO-Vlaanderen through the postdoctoral grant 12T0519N}
\address{Department of Mathematics: Analysis, Logic and Discrete Mathematics\\ Ghent University\\ Krijgslaan 281\\ 9000 Gent\\ Belgium}
\email{andreas.debrouwere@UGent.be}

\author[B. Prangoski]{Bojan Prangoski}
\address{B. Prangoski, Faculty of Mechanical Engineering\\ University Ss. Cyril and Methodius \\ Karpos II bb \\ 1000 Skopje \\ Macedonia}
\email{bprangoski@yahoo.com}


\begin{abstract}
We obtain Gabor frame characterisations of  modulation spaces defined via a class of translation-modulation invariant Banach spaces of distributions that was recently introduced in \cite{D-P-P-V}. We show that these spaces admit an atomic decomposition through Gabor expansions and that they are characterised by summability properties of their Gabor coefficients. Furthermore, we construct a large space of admissible windows. This generalises several fundamental results for the classical modulation spaces $M^{p,q}_{w}$. Due to the absence of solidity assumptions on the Banach spaces defining these modulation spaces, the methods used for the spaces $M^{p,q}_{w}$ (or, more generally, in coorbit space theory) fail in our setting and we develop here a new approach based on the twisted convolution.
\end{abstract}
\maketitle
\section{Introduction}
Modulation spaces, introduced by Feichtinger \cite{Feichtinger1983} in 1983,  are one of the principal objects in time-frequency analysis. Their properties were thoroughly studied by Feichtinger and Gr\"ochenig \cite{Feichtinger1983, F-G1,F-G2,F-G3, Grochenig1991}, often in the more general setting of coorbit spaces. Nowadays, they are widely accepted as an indispensable tool in various branches of analysis; see  e.g.\ \cite{ben-oh,CGN,CR,tof-na}. A key feature of the modulation spaces is that they can be described in a discrete fashion via Gabor frames. Apart from their inherent significance, such characterisations have also turned out to be very useful in applications, e.g., in the study of pseudo-differential operators; see \cite{CR} and the references therein. We refer to the monograph \cite{Grochenig} for an account of results and applications of modulation spaces.

Usually, modulation spaces are defined via weighted mixed-norm spaces \cite{Grochenig} or, more generally -- in the context of coorbit spaces -- via translation invariant solid Banach function spaces \cite{F-G1}. In \cite{D-P-P-V} a  new  class of Banach spaces was proposed to define modulation spaces: A Banach space $F$ is said to be a \emph{translation-modulation invariant Banach space of (tempered) distributions} (TMIB) on $\R^{2n}$ if $F$ satisfies the dense continuous inclusions  $\mathcal{S}(\R^{2n}) \hookrightarrow F \hookrightarrow \mathcal{S}'(\R^{2n})$, $F$ is translation and modulation invariant, and the operator norms of the translation and modulation operators on $F$ are polynomially bounded. We refer to \cite{D-P-P-V} (see also \cite{D-P-V}) for a systematic study of TMIB and their duals (called DTMIB). It is important to point out that TMIB are not necessarily solid (in the sense of \cite{F-G1}).  A natural example of a non-solid TMIB is given by $L^p \widehat{\otimes}_\pi L^p$, $1< p \leq 2$ \cite[Remark 3.10]{D-P-P-V}. The modulation space $\mathcal{M}^F$  associated to a TMIB or DTMIB $F$ on $\R^{2n}$ consists of all those tempered distributions on $\R^n$ whose short-time Fourier transform belongs to $F$.  The basic properties of these spaces were established in \cite[Section 4]{D-P-P-V}.

A natural question that arises is whether modulation spaces defined via TMIB and DTMIB may be described in terms of Gabor frames. The goal of the present paper is to give an affirmative answer to this question, namely, we show that these spaces admit an atomic decomposition through Gabor expansions and that they are characterised by summability properties of their Gabor coefficients. The significance of this lies in the fact that the modulation spaces give a scale of measurement of the regularity and decay properties of tempered distributions and such characterisations allow these properties to be quantified in a discrete way by means of Gabor coefficients.

We now describe the content of this paper in some more detail and point out the main difference between our setting and the classical one involving solid spaces. For $(x,\xi) \in \R^{2n}$, we write $\pi(x,\xi) = M_\xi T_x$, where $T_xf(t) = f(t-x)$ and $M_\xi f(t) = f(t)e^{2\pi i t \cdot \xi}$ denote the translation and modulation operators on $\R^n$. We also set $\check{f}(t) = f(-t)$. Fix a lattice $\Lambda$ in $\R^{2n}$ and a bounded open neighbourhood  $U$ of the origin in $\R^{2n}$ such that the family of sets $\{ \lambda + U \, | \, \lambda \in \Lambda \}$ is pairwise disjoint.  Given a solid TMIB $F$ on $\R^{2n}$, we associate to it the following discrete solid Banach  space on $\Lambda$ \cite[Definition 3.4]{F-G1}
$$
F_d(\Lambda) = \left\{  c = (c_\lambda)_{\lambda \in \Lambda}  \in \C^\Lambda  \, \Big| \,\sum_{\lambda \in \Lambda} c_\lambda  1_{\lambda+U} \in F \right\},
$$
where $1_{\lambda +U}$ is the characteristic function of the set $\lambda + U$, with norm $\| c \|_{F_d(\Lambda) } =  \| \sum_{\lambda \in \Lambda} |c_\lambda|  1_{\lambda+U} \|_F$. The modulation space $\mathcal{M}^F$ admits the following characterisation in terms of Gabor frames \cite{F-G1, Grochenig1991} (see  \cite[Section 12.2]{Grochenig} for the classical modulation spaces $M^{p,q}_{w} =\mathcal{M}^{L^{p,q}_{w}}$).
\begin{theorem}\label{intro} Let $F$ be a solid TMIB on $\R^{2n}$. Set
$$
\omega_F(x,\xi) = \| T_{(x,\xi)} \|_{\mathcal{L}(F)}, \qquad (x,\xi) \in \R^{2n}.
$$
Let $\psi, \gamma \in  M^{1,1}_{\max\{\omega_F,\check{\omega}_F\}}$. Then, the \emph{analysis operator}
$$
C_\psi: \mathcal{M}^F \rightarrow F_d(\Lambda), \, f \mapsto (V_\psi f(\lambda))_{\lambda \in \Lambda}
$$
and the \emph{synthesis operator}
$$
D_\gamma:  F_d(\Lambda) \rightarrow \mathcal{M}^F, \, (c_\lambda)_{\lambda \in \Lambda} \mapsto  \sum_{\lambda \in \Lambda} c_\lambda \pi(\lambda) \gamma
$$
are well-defined and continuous, and the series  $\sum_{\lambda \in \Lambda} c_\lambda \pi(\lambda) \gamma$ is unconditionally convergent in $F$ for each $c \in F_d(\Lambda)$. If in addition $(\psi,\gamma)$ is a pair of dual windows on $\Lambda$, then %
there are $A,B > 0$ such that
$$
A\| f \|_{\mathcal{M}^F} \leq \| (V_\psi f(\lambda))_{\lambda \in \Lambda} \|_{F_{d}(\Lambda)}  \leq B \| f \|_{\mathcal{M}^F}, \qquad f \in \mathcal{M}^F,
$$
and the following expansions hold
$$
f = \sum_{\lambda \in \Lambda} V_\psi f(\lambda) \pi(\lambda) \gamma = \sum_{\lambda \in \Lambda} V_\gamma f(\lambda) \pi(\lambda) \psi, \qquad f \in \mathcal{M}^F,
$$
where both series are unconditionally convergent in $F$.
\end{theorem}
In fact, Theorem \ref{intro} holds for  more irregular samplings sets than lattices \cite{F-G1, Grochenig1991}. The standard proof of Theorem \ref{intro}  \cite{F-G1, Grochenig1991, Grochenig} is based on the following two fundamental properties of the STFT:
\begin{equation}
|V_\psi ( \pi(x,\xi) f)| = |T_{(x,\xi)} V_\psi f|\quad \mbox{and}\quad |V_\psi f| \leq \frac{1}{|(\gamma,\psi)_{L^2}|}  |V_\psi f|\ \ast | V_\psi \gamma|,
\label{eq-solid}
\end{equation}
where $f, \psi, \gamma \in L^2(\R^n)$ with $(\gamma,\psi)_{L^2} \neq 0$; \eqref{eq-solid} may be extended to other spaces. Hence, Theorem \ref{intro} essentially reduces to prove that the mappings
$$
F \to F_d(\Lambda), \, G \mapsto (G \ast \Phi (\lambda))_{\lambda \in \Lambda} \quad \mbox{and} \quad  F_d(\Lambda) \to F, \, (c_\lambda)_{\lambda \in \Lambda} \mapsto   \sum_{\lambda \in \Lambda} c_\lambda T_\lambda \Psi
$$
are well-defined and continuous, and that the series $\sum_{\lambda \in \Lambda} c_\lambda T_\lambda \Psi$ is unconditionally convergent in $F$ for each $c \in F_d(\Lambda)$, where $\Phi,\Psi$ belong to suitable function spaces on $\R^{2n}$.

Our aim is to extend Theorem \ref{intro} to general TMIB.  However, the properties \eqref{eq-solid} are no longer applicable in this setting. The basic idea to overcome this problem is to view the STFT on $L^2(\R^{n})$ as the voice transform of the \emph{projective} representation \cite{C-D-O,Christensen}
$$
\pi : (\R^{2n},+) \rightarrow \mathcal{L}( L^2(\R^n)).
$$
The twisted translation and the twisted convolution associated to $\pi$ are given by
$$
T^\sigma_{(x,\xi)} f(t,\eta) =  f(t-x,\eta- \xi) e^{-2\pi i x\cdot (\eta-\xi)}
$$
and
$$
f \# g(t,\eta) = \iint_{\R^{2n}} f(x,\xi)  g(t-x,\eta- \xi) e^{-2\pi i x\cdot (\eta-\xi)} \dx \dxi.
$$
Then,
\begin{equation}
V_\psi ( \pi(x,\xi) f) = T^\sigma_{(x,\xi)} V_\psi f\quad \mbox{and}\quad V_\psi f =  \frac{1}{(\gamma,\psi)_{L^2} }  V_\psi f \# V_\psi \gamma,
\label{shift-eq}
\end{equation}
where  $f,\psi, \gamma \in L^2(\R^{n}) $ with $(\gamma,\psi)_{L^2} \neq 0$; \eqref{shift-eq} may be extended to other spaces.

From this point of view, it seems natural to define the discrete space associated to a TMIB $F$ via the twisted translation $T^\sigma$, i.e.,
$$
F^\sigma_d(\Lambda) = \left\{  c \in \C^\Lambda  \, \Big| \, \sum_{\lambda \in \Lambda} c_\lambda T^\sigma_\lambda \chi \in F \right\},
$$
where $\chi \in \mathcal{D}(U) \backslash \{0\}$, with norm $\| c \|_{F^B_d(\Lambda)} = \|  \sum_{\lambda \in \Lambda} c_\lambda T^\sigma_\lambda \chi  \|_{F}$. Then, $F^{\sigma}_d(\Lambda)$ is a Banach space that is independent of $\chi \in \mathcal{D}(U) \backslash \{0\}$ (Theorem \ref{basic-sequence}). Moreover, $F_d(\Lambda) =F^\sigma_d(\Lambda)$ if $F$ is solid. We shall  determine the discrete space associated to various TMIB for lattices $\Lambda = \Lambda_1 \times \Lambda_2$, where $\Lambda_1$ and $\Lambda_1$ are lattices in $\R^{n}$ (Subsection \ref{sect-examples}). Most notably,
\begin{gather}
(L^2 \widehat{\otimes}_\pi L^2)^\sigma_d (\Lambda_1 \times \Lambda_2) = \ell^1(\Lambda_1; \ell^2(\Lambda_2)), \label{examples-intro}\\
(L^2 \widehat{\otimes}_\epsilon L^2)^\sigma_d (\Lambda_1 \times \Lambda_2) = c_0(\Lambda_1; \ell^2(\Lambda_2)).\nonumber
\end{gather}

The main results of this paper (Theorem \ref{analysis-coorbit} and Corollary \ref{cor-for-rep11}) show that Theorem \ref{intro} holds for general TMIB $F$ provided that $F_d(\Lambda)$ is replaced by $F^\sigma_d(\Lambda)$, the function $\omega_F$ defining the admissible window class is changed to $\sigma_F$, where
$$
\sigma_F(x,\xi) =  \| T_{(x,\xi)}\|_{\mathcal{L}(F)} \max\{\|M_{(0,x)}\|_{\mathcal{L}(F)},1\},
$$
and the notion of unconditional convergence is weakened to convergence in the C\'esaro sense. Note that $\sigma_F = \omega_F$ if $F$ is solid. Furthermore, an example (Proposition \ref{negative}) shall show that unconditional convergence cannot longer be expected in the setting of TMIB.  We will also prove an analogue of Theorem \ref{intro} for DTMIB.
Similarly as in the solid case, but now by \eqref{shift-eq} instead of \eqref{eq-solid}, the essential problem becomes to show that the mappings
$$
F \rightarrow F^\sigma_d(\Lambda), \, G \mapsto (G \# \Phi (\lambda))_{\lambda \in \Lambda} \quad \mbox{and} \quad  F^\sigma_d(\Lambda) \rightarrow F, \, (c_\lambda)_{\lambda \in \Lambda} \mapsto   \sum_{\lambda \in \Lambda} c_\lambda T^\sigma_\lambda \Psi
$$
are well-defined and continuous, and that the series $\sum_{\lambda \in \Lambda} c_\lambda T_\lambda \Psi$ is C\'esaro summable in $F$ for each $c \in F^\sigma_d(\Lambda)$, where $\Phi,\Psi$ belong to suitable function spaces on $\R^{2n}$.

As an application, we mention that our main results may be used to give explicit descriptions of modulation spaces associated to TMIB and DTMIB. For example, \eqref{examples-intro} implies that $\mathcal{M}^{L^2 \widehat{\otimes}_\pi L^2} = \mathcal{F}M^{2,1}$ (cf.\ Corollary \ref{cor-mod-tensor}). This identity  and various related statements were recently shown in \cite{f-p-p} via different methods. We believe that our work might be used to improve some of the results from \cite{f-p-p}  and we plan to investigate this in the future (see also Problem \ref{tensor-problem}).

The paper is organised as follows. In the preliminary Sections \ref{sect-notation} and \ref{sect-TMIB}, we fix the notation and collect several results concerning TMIB and DTMIB. In Section \ref{sect-twisted}, we define and discuss the twisted translation and the twisted  convolution with respect to a real-valued $n\times n$-matrix; although we are mainly interested in $T^\sigma$ and $\#$, it will turn out that this general setting is technically more convenient. In Section \ref{sect-seq}, the technical core of this paper, we introduce and thoroughly study discrete spaces defined via a twisted translation and associated to a TMIB or DTMIB. Finally, in Section \ref{sect-Gabor}, we show our main results and discuss some applications.

\section{Notation}\label{sect-notation}
We use  standard notation from distribution theory \cite{Schwartz}. For a compact set $K \Subset \R^n$ we denote by $\mathcal{D}_K$ the Fr\'echet space of smooth functions $\varphi$ on $\R^n$ with $\operatorname{supp} \varphi \subseteq K$. Given an open set $U \subseteq \R^n$, we define
$$
\mathcal{D}(U) := \varinjlim_{K \Subset U} \mathcal{D}_K.
$$
We write $\mathcal{S}(\R^n)$ for the Fr\'echet space of rapidly decreasing smooth functions on $\R^n$ and use the following family of norms on  $\mathcal{S}(\R^n)$
$$
\| \varphi \|_{\mathcal{S}^N} := \max_{|\alpha| \leq N} \sup_{x \in \R^n} |\partial^{\alpha}\varphi(x)|(1+|x|)^N, \qquad N \in \N.
$$
The dual spaces $\mathcal{D}'(\R^n)$ and $\mathcal{S}'(\R^n)$ are the space of distributions on $\R^n$ and the space of tempered distributions on $\R^n$, respectively. Unless stated otherwise, we endow these spaces with their strong topology.

The constants in the Fourier transform are fixed as follows
$$
\mathcal{F}(f)(\xi) = \widehat{f}(\xi) := \int_{\R^n} f(x) e^{-2\pi i x\cdot \xi} \dx, \qquad f \in L^1(\R^n).
$$
The Fourier transform is a topological isomorphism from $\mathcal{S}(\R^n)$ onto itself and extends via duality to a topological isomorphism from $\mathcal{S}'(\R^n)$ onto itself. Given a Banach space  $X \subset \mathcal{S}'(\R^n)$, we define its associated Fourier space as the Banach space $\mathcal{F} X := \{ f \in  \mathcal{S}'(\R^n) \, | \, \mathcal{F}^{-1} f \in X \}$ with  norm $\| f \|_{\mathcal{F}X} := \| \mathcal{F}^{-1} f \|_X$.

The translation and modulation operators are defined as $T_xf(t) = f(t-x)$ and $M_\xi f(t) = f(t)e^{2\pi i t \cdot \xi}$,  $x,\xi \in \R^n$. They act continuously on $\mathcal{D}(\R^n)$ and  $\mathcal{S}(\R^n)$, and, by duality, therefore also on $\mathcal{D}'(\R^n)$ and $\mathcal{S}'(\R^n)$. We have that
$$
M_\xi T_x = e^{2\pi i x \cdot \xi} T_x M_\xi, \qquad \mathcal{F} T_x = M_{-x} \mathcal{F}, \qquad \mathcal{F} M_\xi = T_\xi \mathcal{F}.
$$
Furthermore, we write $\check{f}(t) = f(-t)$ for reflection about the origin.\\
\indent Let $\Omega$ be a locally compact, $\sigma$-compact Hausdorff space and let $(\Omega,\Sigma,\mu)$ be a measure space with $\mu$ a positive locally finite Borel measure. A Banach space $E$ is called a \emph{solid Banach function space} on $\Omega$ (cf.\ \cite{F-G1}) if $E \subset L^1_{\operatorname{loc}}(\Omega)$ with continuous inclusion and $E$ satisfies the following condition:
$$
\forall f \in E \, \forall g \in L^1_{\loc}(\Omega) \, : \; |g| \leq |f| \mbox{ a.e.} \Rightarrow g \in E \mbox{ and } \|g\|_E \leq \| f \|_E.
$$
\indent Throughout the article, $C, C', \ldots$ denote absolute constants that may vary from place to place.

\section{Translation-modulation invariant Banach spaces of distributions and their duals}\label{sect-TMIB}
\subsection{Definition and basic properties} We start with the following basic  definition from \cite{D-P-P-V}.
\begin{definition} \label{Def-1} A Banach space $E$ is called a \emph{translation-modulation invariant Banach space of distributions} (TMIB) on $\R^n$ if the following three conditions hold:
\begin{itemize}
\item[$(i)$] $E$ satisfies the dense continuous inclusions $\mathcal{S}(\R^n) \hookrightarrow E \hookrightarrow \mathcal{S}'(\R^n)$.
\item[$(ii)$] $T_x(E) \subseteq E$ and $M_\xi(E) \subseteq E$ for all $x, \xi \in \R^n$.
\end{itemize}
\begin{itemize}
\item[$(iii)$] There exist $\tau_j, C_j > 0$, $j=0,1$, such that
\begin{equation}
\omega_E(x) := \| T_{x} \|_{\mathcal{L}(E)} \leq C_0(1+|x|)^{\tau_0} \quad \mbox{and} \quad \nu_E(\xi) := \| M_{-\xi} \|_{\mathcal{L}(E)} \leq  C_1(1+|\xi|)^{\tau_1};
\label{twf}
\end{equation}
for $x, \xi \in \R^n$ fixed, the mappings $T_x: E \rightarrow E$ and $M_\xi: E \rightarrow E$  are continuous  by the closed graph theorem.
\end{itemize}
\end{definition}
\noindent In what follows, the constants  $\tau_j, C_j > 0$,  $j=0,1$,  will always refer to those occurring in \eqref{twf}.

Let $E$ be a TMIB. Then, $E$ is separable and, for $e \in E$ fixed, the mappings
\begin{equation}
\R^n \rightarrow E, \, x \mapsto T_x e \qquad \mbox{and} \qquad \R^n \rightarrow E, \, \xi \mapsto M_\xi e
\label{orbit}
\end{equation}
are continuous. The functions $\omega_E$ and $\nu_E$ are Borel measurable (as $E$ is separable) and submultiplicative.

An interesting feature of TMIB is that they are stable under taking completed tensor products with respect to the $\pi$- and $\epsilon$-topology \cite{ryan}. Namely, let $E_j$ be a TMIB on $\R^{n_j}$ for $j = 1,2$. Let $\tau$ denote either $\pi$ or $\epsilon$. Then, \cite[Theorem 3.6]{D-P-P-V} (and \cite[Lemma 2.3]{f-p-p} for $\tau = \pi$) yields that $E_1 \widehat{\otimes}_{\tau} E_2$ is a TMIB on $\R^{n_1+n_2}$ with $\omega_{E_1\widehat{\otimes}_{\tau} E_2}=\omega_{E_1}\otimes \omega_{E_2}$ and $\nu_{E_1\widehat{\otimes}_{\tau} E_2}=\nu_{E_1}\otimes \nu_{E_2}$.\\
\indent Next, we introduce dual translation-modulation invariant Banach spaces of distributions \cite{D-P-P-V}.
\begin{definition} \label{Def-2} A Banach space is called a  \emph{dual translation-modulation invariant Banach space of distributions} (DTMIB) on $\R^n$ if it is the strong dual of a TMIB on $\R^n$.
\end{definition}
Let $E$ be a DTMIB. Then, $E$ satisfies the continuous inclusions $\mathcal{S}(\R^n) \rightarrow E \rightarrow \mathcal{S}'(\R^n)$ and the conditions $(ii)$ and $(iii)$ from Definition \ref{Def-1}. If $E = E'_0$, where $E_0$ is a TMIB, then $\omega_E = \check{\omega}_{E_0}$ and $\nu_E = \nu_{E_0}$, whence $\omega_E$ and $\nu_E$ are Borel measurable. Moreover, for $e \in E$ fixed, the mappings in \eqref{orbit} are continuous with respect to  the weak-$\ast$ topology on $E$. In general, $E$ is not a TMIB. More precisely, the inclusion $\mathcal{S}(\R^n) \rightarrow E$ need not be dense and the mappings in \eqref{orbit} may fail to be continuous; consider, e.g., $E = L^\infty$. However, if $E$ is reflexive, then $E$ is in fact a TMIB \cite[Proposition 3.14]{D-P-V} (see also \cite[p.\ 827]{D-P-P-V}).

We now give some examples of TMIB and DTMIB; see also \cite[Section 3]{D-P-P-V}.
\begin{examples}\label{examples-TMIB}
$(i)$  A Banach space $E$ is called a \emph{solid TMIB (DTMIB)} on $\R^n$ if $E$ is both a TMIB (DTMIB) and a solid Banach function space on $\R^n$ (with respect to the Lebesgue measure). Then, $\| M_\xi e \|_E = \| e \|_E$ for all $e \in E$ and $\xi \in \R^n$.
A measurable function $w: \R^n \to (0,\infty)$ is called a \emph{polynomially bounded weight function} on $\R^n$ if there are $C,\tau > 0$ such that
$$
w(x+y) \leq Cw(x)(1+|y|)^{\tau}, \qquad x,y \in \R^n.
$$
For $1 \leq p \leq \infty$ we define $L^p_w = L^p_w(\R^n)$ as the Banach space consisting of all (equivalence classes of) measurable functions $f$ on $\R^n$ such that $\| f\|_{L^p_w} := \| f w \|_{L^p} < \infty$. We define $C_{0,w} = C_{0,w}(\R^n)$ as the closed subspace of $L^{\infty}_w$ consisting of all $f\in C(\R^n)$ such that $\lim_{|x| \to \infty} f(x)w(x) = 0$. Then, $L^p_w$, $1 \leq p < \infty$, is a solid TMIB, $L^{p}_w$, $1 < p \leq \infty$, is a solid DTMIB, and $C_{0,w}$ is a TMIB. Similarly, we may consider weighted mixed-norm spaces. Let  $w$  be a polynomially bounded weight function on $\R^{n_1+n_2}$. For $1 \leq p_1, p_2 \leq \infty$ we define $L^{p_1,p_2}_{w} = L^{p_1,p_2}_{w}(\R^{n_1+n_2})$ as the Banach space consisting of all (equivalence classes of) measurable functions $f$ on $\R^{n_1+n_2}$ such that $\| f\|_{L^{p_1,p_2}_{w}} := \| f w \|_{L^{p_1,p_2}}$. Then, $L^{p_1,p_2}_{w}$ is a solid TMIB if $1 \leq p_1, p_2 < \infty$ and a solid DTMIB if $1 < p_1, p_2 \leq \infty$.
\\
\noindent $(ii)$ Let $E$ be a TMIB (DTMIB). Then, $\mathcal{F}E$ is a TMIB (DTMIB) with $\omega_{\mathcal{F}E} = \check{\nu}_E$ and $\nu_{\mathcal{F}E} = \omega_E$. If $E$ is solid, we have that $\| T_x e \|_{\mathcal{F}E} = \| e \|_{\mathcal{F}E}$   for all $e \in \mathcal{F}E$ and $x \in \R^n$. The Sobolev spaces $\mathcal{F}L^p_w$, with $w$  a polynomially bounded weight function on $\R^{n}$, are of this type.
\\
\noindent $(iii)$  Let $w$ be a polynomially bounded weight function on $\R^{n_1}$ and let $E$ be a TMIB on $\R^{n_2}$. Then, the weighted Bochner-Lebesgue space $L^p_w(E) = L^p_w(\R^{n_1};E)$, $1 \leq p < \infty$, and the weighted vector-valued $C_0$-space $C_{0,w}(E) = C_{0,w}(\R^{n_1};E)$ are TMIB on $\R^{n_1+n_2}$. If $E'$ satisfies the Radon-Nikod\'ym property (in particular, if $E$ is reflexive), then (cf. \cite[Theorem 3.5]{Chaney})
$$
L^p_w(E') = (L^q_{1/w}(E))', \qquad 1 < p \leq \infty,
$$
where $q$ denotes the H\"older conjugate index to $p$. In particular, $L^p_w(E')$, $1< p \leq \infty$, is a DTMIB on $\R^{n_1+n_2}$. 
\\
\noindent $(iv)$ The spaces $L^{p_1}(\R^{n_1})  \widehat{\otimes}_{\pi} L^{p_2}(\R^{n_2})$ and  $L^{p_1}(\R^{n_1})  \widehat{\otimes}_{\epsilon} L^{p_2}(\R^{n_2})$, $1 \leq p_1, p_2 < \infty$, are TMIB on $\R^{n_1+n_2}$ consisting of locally integrable functions. In \cite[Remark 3.10]{D-P-P-V} it is shown that $L^{p}(\R^{n})  \widehat{\otimes}_\pi L^{p}(\R^{n})$, $1 < p \leq 2$, is not solid.
\end{examples}

\subsection{Convolution and multiplication}\label{subs-conv-mult}
Every TMIB or DTMIB $E$ is a Banach convolution module over the Beurling algebra $L^1_{\omega_E}$ and a Banach multiplication module over the Wiener-Beurling algebra $\mathcal{F}L^1_{\nu_E}$\footnote{The Wiener-Beurling algebra $\mathcal{F}L^1_{\nu_E}$ is sometimes denoted as $A_{\nu_E}$.}. More precisely, if $E$ is a TMIB, the convolution $\ast : \mathcal{S}(\R^n) \times \mathcal{S}(\R^n)  \rightarrow \mathcal{S}(\R^n)$ and multiplication $\cdot : \mathcal{S}(\R^n) \times \mathcal{S}(\R^n)  \rightarrow \mathcal{S}(\R^n)$ extend uniquely to continuous bilinear mappings $\ast : E \times L^1_{\omega_E}  \rightarrow E$ and  $\cdot : E \times \mathcal{F}L^1_{\nu_E}  \rightarrow E$ such that
\begin{equation}
\|e \ast f \|_E\leq \|e\|_E \|f\|_{L^1_{\omega_E}}, \qquad e \in E,\, f \in L^1_{\omega_E},
\label{norm-1}
\end{equation}
and
\begin{equation}
\|e\cdot f\|_E\leq \|e\|_E \|f\|_{\mathcal{F}L^1_{\nu_E}}, \qquad e \in E,\, f \in \mathcal{F}L^1_{\nu_E}.
\label{norm-2}
\end{equation}
Moreover, the following integral representations hold
\begin{equation}
e \ast f=\int_{\R^n} T_xef(x) \dx,\qquad  e \in E,\, f \in L^1_{\omega_E},
\label{int-1}
\end{equation}
and
\begin{equation}
e \cdot f =\int_{\R^n} M_{-x}e \mathcal{F}^{-1}f(x) \dx,  \qquad e \in E,\, f \in \mathcal{F}L^1_{\nu_E},
\label{int-22}
\end{equation}
where the integrals should be interpreted as  $E$-valued Bochner integrals  \cite[Proposition 3.2]{D-P-P-V}. Next, suppose  that $E$ is a DTMIB with $E = E_0'$, where $E_0$ is a TMIB. The convolution and multiplication on $E$ are defined via duality, namely, for $e \in E$, $f \in L^1_{\omega_E}$, and $g \in \mathcal{F}L^1_{\nu_E}$, we set
$$
\langle e \ast f, g \rangle := \langle e, g \ast \check{f} \rangle, \qquad g \in E_0,
$$
and
$$
\langle e \cdot f, g \rangle := \langle e, g \cdot f \rangle, \qquad g \in E_0.
$$
Then, the inequalities \eqref{norm-1} and \eqref{norm-2} hold true and the integral representations \eqref{int-1} and \eqref{int-22} are valid if the integrals are interpreted as $E$-valued Pettis integrals with respect to the weak-$\ast$ topology on $E$  \cite[Corollary 3.5]{D-P-P-V}. Hence, TMIB and DTMIB may be viewed as Banach spaces of distributions having two module structures in the sense of \cite{brfe83}.\\
\indent The goal of this subsection is to extend the previous results by showing that TMIB and DTMIB are in fact  Banach convolution and multiplication  modules  over a certain weighted space of Radon measures and its associated Fourier space, respectively. Our approach is based on the integral representations \eqref{int-1} and \eqref{int-22}. The following lemma will allow us to treat TMIB and DTMIB simultaneously. Its proof is standard and therefore we omit it.

\begin{lemma}\label{weak-Bochner} Let $X_0$ be a separable Banach space and set $X = X'_0$. Let $(\Omega,\Sigma,\mu)$ be a  measure space with $\mu$ a complex measure.  Let ${\bf f}: \Omega \rightarrow X$ be weak-$\ast$  measurable, i.e., the function $\Omega \rightarrow \C, \, x \mapsto \langle \mathbf{f} (x), g \rangle$ is measurable for every $g \in X_0$. Furthermore, suppose that
 \begin{equation}
\int_{\Omega} \| \mathbf{f} (x)  \|_X {\rm d}|\mu|(x) < \infty.
\label{w-int-1}
\end{equation}
Then, ${\bf f}: \Omega \rightarrow X$ is Pettis integrable with respect to the weak-$\ast$ topology on $X$ and
\begin{equation}
\left \|\int_{\Omega} \mathbf{f} (x) {\rm d}\mu(x) \right \|_X   \leq  \int_{\Omega} \| \mathbf{f} (x)  \|_X {\rm d}|\mu|(x).
\label{w-int-2}
\end{equation}
\end{lemma}
\noindent We will use Lemma \ref{weak-Bochner} without explicitly referring to it.

Let $\omega: \R^n \rightarrow [1,\infty)$ be a Borel measurable submultiplicative polynomially bounded function. We denote by $\mathcal{M}^1_{\omega} =\mathcal{M}^1_{\omega}(\R^n)$ the Banach space consisting of all complex Radon measures $\mu$ on $\R^n$ such that $\| \mu\|_{\mathcal{M}^1_{\omega}}  := \int_{\R^n} \omega(x) {\rm d}|\mu|(x) < \infty$. The space $\mathcal{M}^1_{\omega} \subset \SSS'(\R^n)$ is a Banach convolution module and its associated Fourier space $\mathcal{F}\mathcal{M}^1_{\omega}$ is a Banach multiplication module if the multiplication is defined via the Fourier transform and the convolution in $\mathcal{M}^1_{\omega}$. Since $\mathcal{M}^1_{\omega} \subseteq \mathcal{M}^1$, the elements of $\mathcal{F}\mathcal{M}^1_{\omega}$ are bounded continuous functions and the multiplication defined above coincides with the ordinary multiplication of continuous functions.

Let $E$ be a TMIB or a DTMIB and set $\widetilde{\omega}_E=\max\{1,\omega_E\}$. We define the convolution of $e \in E$ and $\mu \in \mathcal{M}^1_{\widetilde{\omega}_E}$ as
$$
e \ast \mu := \int_{\R^n} T_x e \,  {\rm d}\mu(x) \in E,
$$
where the integral should be interpreted as an $E$-valued Bochner integral if $E$ is a TMIB and as an $E$-valued Pettis integral with respect to the weak-$\ast$ topology on $E$ if $E$ is a DTMIB; hereafter, for DTMIB $E$, $E$-valued Pettis integrals will always be meant with respect to the weak-$\ast$ topology on $E$ (cf.\ Lemma \ref{weak-Bochner}). Hence, $\ast: E \times \mathcal{M}^1_{\widetilde{\omega}_E} \rightarrow E$ is a continuous bilinear mapping such that
$$
\| e \ast \mu \|_{E} \leq \|e\|_{E} \|\mu\|_{\mathcal{M}^1_{\widetilde{\omega}_E}}, \qquad e \in E,\, \mu \in \mathcal{M}^1_{\widetilde{\omega}_E}.
$$
If ${\rm d}\mu(x)=f(x)\dx$ with $f\in L^1_{\widetilde{\omega}_E}$, this definition of convolution coincides with the one given at the beginning of the subsection. Furthermore, if
$$
\int_{\R^n} (1+|x|)^N {\rm d}|\mu|(x)<\infty, \qquad \forall N \in \N,
$$
then $\mu \in \mathcal{O}'_C(\R^n)$  \cite[p.\ 244]{Schwartz} and
$$
\langle e \ast \mu, \varphi \rangle= \langle e, \varphi \ast \check{\mu} \rangle, \qquad \varphi \in \mathcal{S}(\R^n),
$$
whence $e \ast \mu$ is equal to the $\mathcal{S}'(\R^n) \times \mathcal{O}'_C(\R^n)$-convolution of $e$ and $\mu$ \cite[Theor\`eme XI, p.\ 247]{Schwartz}. Next, we consider multiplication. Set $\widetilde{\nu}_E =  \max \{ 1, \nu_E\}$.  We define the multiplication  of $e \in E$ and $f \in \mathcal{F}\mathcal{M}^1_{ \widetilde{\nu}_E}$ as
$$
e \cdot f := \int_{\R^n} M_{-x} e \,  {\rm d} \mathcal{F}^{-1}f(x);
$$
the integral should be interpreted as an $E$-valued Bochner integral if $E$ is a TMIB and as an $E$-valued Pettis integral if $E$ is a DTMIB. Hence, $\cdot: E \times \mathcal{F}\mathcal{M}^1_{ \widetilde{\nu}_E} \rightarrow E$ is a continuous bilinear mapping such that
$$
\| e \cdot f \|_{E} \leq \|e\|_{E} \|f\|_{\mathcal{F}\mathcal{M}^1_{\widetilde{\nu}_E}}, \qquad e \in E,\, f \in \mathcal{F} \mathcal{M}^1_{ \widetilde{\nu}_E}.
$$
If $f\in \mathcal{F}L^1_{\widetilde{\nu}_E}$, this definition of multiplication coincides with the one given at the beginning of the subsection. Furthermore, if
$$
\int_{\R^n} (1+|x|)^N {\rm d} |\mathcal{F}^{-1}f| (x)<\infty, \qquad \forall N \in \N,
$$
then $f \in \mathcal{O}_M(\R^n)$  \cite[p.\ 243]{Schwartz} and
$$
\langle e \cdot f, \varphi \rangle= \langle e, \varphi \cdot f  \rangle, \qquad \varphi \in \mathcal{S}(\R^n),
$$
whence $e \cdot f$ is equal to the $\mathcal{S}'(\R^n) \times \mathcal{O}_M(\R^n)$-multiplication of $e$ and $f$ \cite[Theor\`eme X, p.\ 246]{Schwartz}. Suppose that $E$ is a DTMIB with $E = E_0'$, where $E_0$ is a TMIB. For $e \in E$, $\mu \in \mathcal{M}^1_{\widetilde{\omega}_E}$, and  $f \in \mathcal{F} \mathcal{M}^1_{\widetilde{\nu}_E}$ it holds that
$$
\langle e \ast \mu, g \rangle = \langle e, g \ast \check{\mu} \rangle\quad \mbox{and} \quad \langle e \cdot f, g \rangle = \langle e, g \cdot f \rangle, \qquad g \in E_0.
$$

Every solid Banach function space  is a Banach multiplication module over $L^\infty$. We now use the previous observations to formulate a result that, for our purposes, will turn out to be the suitable analogue of this fact for TMIB and DTIMB.
We first need to introduce some  terminology.  A lattice $\Lambda$  is a discrete subgroup of $\R^n$ that spans the real vector space $\R^n$. There is a unique invertible $n\times n$-matrix $A_\Lambda$ such that $\Lambda = A_\Lambda \Z^n$. The dual lattice of $\Lambda$ is defined as $\Lambda^\perp = (A_\Lambda^t)^{-1} \Z^n =  \{ \mu \in \R^n \, | \, \lambda \cdot \mu \in \Z, \, \forall \lambda \in \Lambda \}$. We define $ I_\Lambda := A_\Lambda[0,1)^n$ and $\operatorname*{vol}(\Lambda) := | I_\Lambda |  =  | \det A_\Lambda|$.

\begin{lemma}\label{mea-spa-sub}
Let $\omega: \R^n \rightarrow [1,\infty)$ be a Borel measurable submultiplicative polynomially bounded function. Let $\Lambda$ be a lattice in $\R^n$. Then, for every $y\in\R^n$, the bilinear mapping
$$
\mathcal{F}L^1_{\omega}\times \SSS(\R^n)\rightarrow \mathcal{F}\mathcal{M}^1_{\omega},\, (f,\varphi)\mapsto \sum_{\lambda\in\Lambda}e^{2\pi i y\cdot\lambda}T_{\lambda}(f \varphi),
$$
is well-defined and continuous. Furthermore, there are $C > 0$ and $N \in \N$ such that
$$
\sup_{y \in \R^n}\left\|\sum_{\lambda\in\Lambda}e^{2\pi i y\cdot\lambda}T_{\lambda}(f \varphi)\right\|_{\mathcal{F}\mathcal{M}^1_{\omega}}\leq C\|f\|_{\mathcal{F}L^1_{\omega}} \|\varphi\|_{\SSS^N},\quad f\in\mathcal{F}L^1_{\omega},\,\varphi\in\SSS(\R^n).
$$
\end{lemma}
\begin{proof}
Let $f\in\mathcal{F}L^1_{\omega}$, $\varphi\in\SSS(\R^n)$, and $y \in \R^n$ be arbitrary. The Poisson summation formula implies that
\begin{equation}
\mathcal{F}^{-1}\left(\sum_{\lambda\in\Lambda}e^{2\pi i y\cdot \lambda}T_{\lambda}(f \varphi)\right)= \frac{1}{\operatorname*{vol}(\Lambda)} \sum_{\mu\in\Lambda^{\perp}} \widehat{f\varphi}(\mu+y)T_{-\mu-y}\delta \quad \mbox{in } \SSS'(\R^n).
\label{identity-cont}
\end{equation}
Hence,
\begin{align*}
\left\|\sum_{\lambda\in\Lambda}e^{2\pi i y\cdot\lambda} T_{\lambda}(f\varphi)\right\|_{\mathcal{F}\mathcal{M}^1_{\omega}}&= \frac{1}{\operatorname*{vol}(\Lambda)} \sum_{\mu\in\Lambda^{\perp}}|\widehat{f}\ast \widehat{\varphi}(\mu+y)|\check{\omega}(\mu+y)\\
&\leq \frac{1}{\operatorname*{vol}(\Lambda)} \int_{\R^n}|\widehat{f}(y-x)|\check{\omega}(y-x)\sum_{\mu\in\Lambda^{\perp}} |\widehat{\varphi}(\mu+x)|\check{\omega}(\mu+x)\dx\\
&\leq C\|\widehat{\varphi}\|_{L^{\infty}_{(1+|\cdot|)^{n+1}\check{\omega}}} \|\mathcal{F}f\|_{L^1_{\check{\omega}}}.
\end{align*}
As the Fourier transform is an isomorphism from $\mathcal{S}(\R^n)$ onto itself and $\|\mathcal{F}f\|_{L^1_{\check{\omega}}}= \|f\|_{\mathcal{F}L^1_{\omega}}$, this completes the proof.
\end{proof}

\begin{corollary}\label{mult-new}
Let $\Lambda$ be a lattice in $\R^n$ and let $\varphi \in \mathcal{S}(\R^n)$.
Then, for every $y\in\R^n$, the bilinear mapping
$$
E \times \mathcal{F}L^1_{\widetilde{\nu}_E} \rightarrow E,\, (e,f) \mapsto e \cdot \sum_{\lambda\in\Lambda}e^{2\pi i y\cdot\lambda}T_{\lambda}(f \varphi),
$$
is well-defined and continuous. Furthermore, there is $C > 0$ such that
$$
\sup_{y \in \R^n}\left\| e \cdot \sum_{\lambda\in\Lambda}e^{2\pi i y\cdot\lambda}T_{\lambda}(f \varphi)\right\|_{E} \leq C \| e \|_E \|f\|_{\mathcal{F}L^1_{\widetilde{\nu}_E}}, \qquad  e \in E,\, f\in\mathcal{F}L^1_{\widetilde{\nu}_E}.
$$
\end{corollary}

\subsection{Amalgam spaces}\label{sub-section-amalgam}

In this subsection, we define  amalgam spaces which have a TMIB or a DTMIB as local component. These spaces will play an important technical role in the rest of this article. We refer to \cite{feich,f-p-p} for more information.\\
\indent Let $E$ be a TMIB or DTMIB. We define $E_{\operatorname{loc}}=\{f\in\DD'(\R^n)\,|\, \chi f\in E,\, \forall \chi\in\DD(\R^n)\}$. Since $\mathcal{D}(\R^n) \subset \mathcal{F}L^1_{\nu_E}$, the function $\R^n \to E$, $x\mapsto f T_x\chi$ is continuous for all $f \in E_{\operatorname{loc}}$ and $\chi \in \mathcal{D}(\R^n)$. Let $w$ be a polynomially bounded weight function on $\R^n$ and let $1\leq p \leq \infty$. Fix $\chi \in \mathcal{D}(\R^n) \backslash \{0\}$. We define the amalgam space $W(E,L^p_w)$ as the space consisting of all $f \in  E_{\operatorname{loc}}$ such that (cf. \cite{feich}, \cite[Section 3]{f-p-p})
$$
\| f \|_{W(E,L^p_w)} :=  \left(\int_{\R^n}\|f T_x\chi\|_E^p w(x)^p\dx\right)^{1/p}<\infty
$$
(with the obvious modification for $p=\infty$). Then, $W(E,L^p_w)$ is a Banach space whose definition  is independent of the choice $\chi \in \mathcal{D}(\R^n) \backslash \{0\}$ and different non-zero elements of $\mathcal{D}(\R^n)$ induce equivalent norms on  $W(E,L^p_w)$ (cf.\ \cite[Lemma 3.4]{f-p-p}, \cite[Theorem 1]{feich}). By  \cite[Lemma 3.2]{f-p-p}, $W(E,L^p_w)$, $1\leq p<\infty$, is a TMIB if $E$ is so, while $W(E,L^p_w)$, $1 \leq p\leq \infty$, is a DTMIB if $E$ is so.

\section{The twisted translation an the twisted convolution}\label{sect-twisted}
Fix a real-valued $n\times n$-matrix $B$. For $x \in \R^n$ we define \emph{the twisted translation with respect to $B$} as
$$
T^B_x f(t) :=  T_x M_{-Bx} f (t) =  f(t-x)e^{-2\pi i Bx \cdot (t-x)}, \qquad f \in \mathcal{D}'(\R^n).
$$
Note that $T^0_x = T_x$. For all $x,y \in \R^n$, $f \in \mathcal{D}'(\R^n)$, and $\varphi \in C^\infty(\R^n)$ it holds that
\begin{itemize}
\item[$(i)$] $T^B_x  T^{B^t}_y = T^{B^t}_y  T^B_x$.
\item[$(ii)$] $T^B_x(f\cdot \varphi) = T^B_xf\cdot T_x\varphi = T_xf\cdot T^B_x\varphi$.
\item[$(iii)$] $T^B_xf\cdot T^{-B}_x\varphi = T_x(f\cdot \varphi)$.
\end{itemize}
We define \emph{the twisted convolution with respect to $B$}  of $f, g \in L^1$ as
$$
f \ast_B g(t) := \int_{\R^n} f(x) T^B_xg(t) \dx.
$$
Note that $f \ast^0 g = f \ast g$. Define
$$
\theta_B(f)(x) := e^{2\pi i Bx \cdot x} f(x), \qquad f \in \mathcal{D}'(\R^n).
$$
For all  $f, g \in L^1$, $h\in L^1\cap L^{\infty}$, it holds that
\begin{itemize}
\item[$(i)$] $f \ast_B g = g \ast_{B^t} f$.
\item[$(ii)$] $\displaystyle f \ast_B g(t) =  \int_{\R^n} f(x) T^{-B}_t (\theta_B(\check{g}))(x) \dx$.
\item[$(iii)$] $\displaystyle \int_{\R^n} f \ast_B g(t) h(t) \dt = \int_{\R^n} f(t) h \ast_{-B} \theta_B(\check{g})(t) \dt$.
\end{itemize}
\begin{definition} \label{def-B0}
Consider the real-valued $2n \times 2n$-matrix
$$
B_0 := \begin{pmatrix}
0 & 0 \\
I & 0
\end{pmatrix}
.$$
Following the notation used in the introduction, we set
$$
T^\sigma_{(x,\xi)} f(t,\eta) := T^{B_0}_{(x,\xi)}f(t,\eta) = f(t-x,\eta- \xi) e^{-2\pi i x\cdot (\eta-\xi)}, \qquad (x,\xi) \in \R^{2n},
$$
and  for $f, g \in L^1(\R^{2n})$
$$
f \# g(t,\eta) := f \ast_{B_0} g(t,\eta) =  \iint_{\R^{2n}} f(x,\xi)  g(t-x,\eta- \xi) e^{-2\pi i x\cdot (\eta-\xi)} \dx \dxi.
$$
\end{definition}

Next, we  extend the twisted convolution to $\mathcal{S}'(\R^n)$. The proof of the following lemma is straightforward and we omit it.
\begin{lemma} \label{twisted-test} \mbox{}
\begin{itemize}
\item[$(i)$] The mapping  $T^B_x: \mathcal{S}(\R^n) \rightarrow \mathcal{S}(\R^n)$ is  continuous for each $x \in \R^n$. More precisely,
$$
\| T^B_x \varphi \|_{\mathcal{S}^N} \leq (1+2\pi \|B\|)^N \| \varphi \|_{\mathcal{S}^N}  (1+|x|)^{2N}, \qquad  \varphi \in \mathcal{S}(\R^n),\, N \in \N,
$$
 where $\|B\|$ denotes the operator norm of $B$.
\item[$(ii)$] The mapping  $\R^n \rightarrow \mathcal{S}(\R^n), \, x \mapsto T^B_x \varphi$ is continuous for each $\varphi \in  \mathcal{S}(\R^n)$.
\item[$(iii)$] The mappings $\theta_B: \mathcal{S}(\R^n) \rightarrow \mathcal{S}(\R^n)$ and $\theta_B: \mathcal{S}'(\R^n) \rightarrow \mathcal{S}'(\R^n)$ are continuous.
\item[$(iv)$] The bilinear mapping $\ast_B:  \mathcal{S}(\R^n) \times  \mathcal{S}(\R^n) \rightarrow   \mathcal{S}(\R^n)$ is continuous.
\end{itemize}
\end{lemma}
We define the twisted convolution of $f \in \mathcal{S}'(\R^n)$ and $\varphi \in \mathcal{S}(\R^n)$ as
\begin{equation}
f \ast_B \varphi(x) := \langle f, T^{-B}_x\theta_B(\check{\varphi}) \rangle.
\label{def-two}
\end{equation}
Then, $f \ast_B \varphi \in C(\R^n)$ and  $\| f \ast_B \varphi \|_{L^\infty_{(1+|\, \cdot \, |)^{-N}}} < \infty$ for some $N \in \N$. If $A \subset \mathcal{S}'(\R^n)$ is bounded, the previous estimate holds uniformly for $f \in A$. Since $\mathcal{S}'(\R^n)$ is bornological, this implies that the mapping
$$
 \mathcal{S}'(\R^n) \rightarrow \mathcal{S}'(\R^n), \, f \mapsto f \ast_B \varphi
$$
is continuous. As $L^1$ is dense in $\mathcal{S}'(\R^n)$, we have that
$$
\langle f \ast_B \varphi, \psi \rangle = \langle f, \psi \ast_{-B} \theta_B(\check{\varphi}) \rangle, \qquad \psi \in  \mathcal{S}(\R^n),
$$
for all $f \in \mathcal{S}'(\R^n)$ and $\varphi \in \mathcal{S}(\R^n)$.

Finally, we discuss the twisted convolution on TMIB and DTMIB. Let $E$ be a TMIB or a DTMIB. Then, $T^B_x: E \rightarrow E$ is continuous for each $x \in \R^n$ and
\begin{equation}
\rho^B_E(x) := \| T^{B}_x\|_{\mathcal{L}(E)} \leq \omega_E(x) \nu_E(Bx) \leq C_2(1+|x|)^{\tau_0+ \tau_1},
\label{twisted-wf}
\end{equation}
where $C_2 = C_0C_1 \max\{1, \|B\|^{\tau_1}\}$. Note that $\rho^B_E$ is submultiplicative and polynomially bounded. For $e \in E$ fixed, the mapping
\begin{equation}
\R^n \rightarrow E, \, x \mapsto T^{B}_x e,
\label{twisted-orbit}
\end{equation}
is continuous  if $E$ is a TMIB and continuous with respect to the weak-$\ast$ topology on $E$ if $E$ is a DTMIB. Consequently, $\rho^B_E$ is Borel measurable when $E$ is a TMIB (as $E$ is separable). If $E$ is a DTMIB with $E=E'_0$, where $E_0$ a TMIB, the bipolar theorem yields that $\rho^B_E=\check{\rho}^{-B}_{E_0}$, whence $\rho^B_E$ is Borel measurable in this case as well.

Given a Banach space $X\subset \SSS'(\R^n)$, we define the Banach spaces $\check{X}:=\{f\in\SSS'(\R^n)\,|\, \check{f}\in X\}$ with norm $\|f\|_{\check{X}}:=\|\check{f}\|_X$ and $\theta_B X = \{f\in\SSS'(\R^n)\,|\,\theta_{-B}f\in X\}$ with norm $\|f\|_{\theta_B X}:=\|\theta_{-B}f\|_X$.  Furthermore, given a polynomially bounded weight function $w$ on $\R^n$, we denote by $C_w = C_w(\R^n)$ the space $L^\infty_w \cap C(\R^n)$; of course, it is a closed subspace of $L^\infty_w$.\\
\indent Assume that $E$ is a TMIB. The twisted convolution of $e \in E$ and $g \in  (\theta_{-B}E')\check{}$ is defined as
$$
e \ast_B g(x) := {}_E\langle e, T_x^{-B}\theta_B(\check{g})\rangle_{E'}.
$$
Similarly,  we define the twisted convolution of $e \in E'$ and $g \in  (\theta_{-B}E)\check{}$ as
$$
e \ast_B g(x) := {}_{E'}\langle e, T_x^{-B}\theta_B(\check{g})\rangle_{E}.
$$
Obviously, these definitions coincide with the one given in \eqref{def-two} if $g \in \mathcal{S}(\R^n)$. Note that the bilinear mappings
$$
\ast_B :E \times (\theta_{-B}E')\check{}\rightarrow C_{1/\check{\rho}^{B}_{E}}\quad \mbox{and}\quad \ast_B :E'\times (\theta_{-B}E)\check{}\rightarrow C_{1/\check{\rho}^{B}_{E'}}
$$
are well-defined and continuous.

\section{Discrete spaces associated to TMIB and DTMIB}\label{sect-seq}

Throughout this section, $E$ always stands for a TMIB or a DTMIB. We also fix a real-valued $n \times n$-matrix $B$, a lattice $\Lambda$ in $\R^n$, and a bounded open neighbourhood $U$ of the origin such that the family of sets $\{ \lambda + U \, | \, \lambda \in \Lambda \}$ is pairwise disjoint.

\subsection{Definition and basic properties}  The following fundamental definition is inspired by \cite[Definition 3.4]{F-G1}, where  a discrete space is associated to a solid Banach function space.

\begin{definition}\label{def-of-seq-spa}
Let $\chi \in \mathcal{D}(U) \backslash \{ 0\}$. We define \emph{the discrete space associated to $E$ with respect to $B$} as
$$
E^B_d(\Lambda) =E^B_{d,\chi}(\Lambda) := \left\{  c = (c_\lambda)_{\lambda \in \Lambda} \in \C^\Lambda  \, \Big| \,S_\chi(c) := \sum_{\lambda \in \Lambda} c_\lambda T^B_\lambda \chi \in E \right\}
$$
and endow it with the norm $\| c \|_{E^B_d(\Lambda)} = \| c \|_{E^B_{d,\chi}(\Lambda)} := \left \|  S_\chi(c) \right \|_E$.
\end{definition}

We start by showing that $E^B_d(\Lambda)$ is a Banach space whose definition is independent of  $\chi \in \mathcal{D}(U) \backslash \{ 0\}$.
\begin{theorem} \label{basic-sequence} \mbox{}
\begin{itemize}
\item[$(i)$] $E^B_d(\Lambda)$ is a Banach space.
\item[$(ii)$] The definition of $E^B_d(\Lambda)$ is independent of the choice $\chi \in \mathcal{D}(U) \backslash \{ 0\}$ and different  non-zero elements of $\mathcal{D}(U)$ induce equivalent norms on  $E^B_d(\Lambda)$.
\end{itemize}
\end{theorem}
\begin{proof}
$(i)$  Let $(c_j)_{j \in \N}$ be a Cauchy sequence in $E^B_d(\Lambda)$.  Since $E$ is continuously included in $\mathcal{D}'(\R^n)$, the inclusion mapping $E^B_d(\Lambda) \rightarrow \C^\Lambda$ is continuous. Hence, there is $c \in \C^\Lambda$ such that $\lim_{j \to \infty}c_j = c$ in $\C^\Lambda$. As $(S_\chi(c_j))_{j \in \N}$ is a Cauchy sequence in $E$, there is $e \in E$ such that $\lim_{j \to \infty} S_\chi(c_j) = e$ in $E$. Note that $e = S_\chi(c)$ in  $\mathcal{D}'(\R^n)$. Therefore, $c \in E^B_d(\Lambda)$ and $\lim_{j \to \infty}c_j = c$ in $E^B_d(\Lambda)$.

$(ii)$ We divide the proof into three steps.

STEP I: \emph{Let $\widetilde{\chi} \in \mathcal{D}(U) \backslash \{0\}$ be such that $\widetilde{\chi} = 1$ on some non-empty open subset $V$ of $U$. Then, $E^B_{d,\widetilde{\chi}}(\Lambda)$ is continuously included into $E^B_{d,\chi}(\Lambda)$ for all $\chi \in \mathcal{D}(U) \backslash \{0\}$.}

Let $x_0 \in U$ and $r > 0$ be such that $[x_0-r,x_0+r]^n \subset V$. Pick $\psi \in \mathcal{D}_{[-r,r]^n}$ such that $\sum_{m \in \Z^n} T_{rm}\psi = 1$ on $\R^n$. Hence, there is $N \in \N$ such that $\sum_{ |m| \leq N} T_{rm}\psi = 1$ on $\operatorname{supp} \chi$. For all $c \in E^B_{d,\widetilde{\chi}}(\Lambda)$ it holds that
\begin{align*}
\sum_{\lambda \in \Lambda} c_\lambda T^B_\lambda \chi &= \sum_{ |m| \leq N} \sum_{\lambda \in \Lambda} c_\lambda T^B_\lambda( \chi T_{rm}\psi) \\
&= \sum_{ |m| \leq N} \sum_{\lambda \in \Lambda} c_\lambda e^{2\pi i B\lambda\cdot (x_0-rm)} T_{rm-x_0}T^B_\lambda (T_{x_0}\psi T_{x_0-rm}\chi) \\
&=\sum_{ |m| \leq N}  T_{rm-x_0}\left ( \sum_{\lambda \in \Lambda} c_\lambda e^{2\pi i B^t(x_0-rm)\cdot \lambda} T^B_\lambda ( \widetilde{\chi} T_{x_0}\psi T_{x_0-rm}\chi) \right) \\
&=\sum_{ |m| \leq N}T_{rm-x_0} \left ( \sum_{\lambda \in \Lambda} c_\lambda T^B_\lambda \widetilde{\chi} \cdot \sum_{\lambda' \in \Lambda} e^{2\pi i B^t(x_0-rm)\cdot \lambda'} T_{\lambda'}(T_{x_0-rm}\chi T_{x_0}\psi) \right).
\end{align*}
The result is therefore a consequence of Corollary \ref{mult-new}.

STEP II: \emph{ Let $\chi \in \mathcal{D}(U) \backslash \{0\}$. Choose $\widetilde{\chi} \in \mathcal{D}(U)$ such that $\operatorname{supp}\widetilde{\chi}  \subset \{ x \in U \, | \, \chi(x) \neq 0 \}$ and $\widetilde{\chi} = 1$ on some non-empty open subset $V$ of $U$. Then,  $E^B_{d,\chi}(\Lambda) = E^B_{d,\widetilde{\chi}}(\Lambda)$ with equivalent norms.}

By STEP I, $E^B_{d,\widetilde{\chi}}(\Lambda)$ is continuously included in $E^B_{d,\chi}(\Lambda)$. We now show the converse inclusion. Set $\varphi =  \widetilde{\chi} / \chi  \in \mathcal{D}(U)$. For all  $c \in E^B_{d,\chi}(\Lambda)$ it holds that
$$
\sum_{\lambda \in \Lambda} c_\lambda T^B_\lambda \widetilde{\chi} =  \sum_{\lambda \in \Lambda} c_\lambda T^B_\lambda( \chi \varphi)
=  \sum_{\lambda \in \Lambda} c_\lambda T^B_\lambda \chi \cdot \sum_{{\lambda'} \in \Lambda} T_{\lambda'} \varphi,
$$
whence the result follows from Corollary \ref{mult-new}.

STEP III: \emph{ Let $\chi_1, \chi_2 \in \mathcal{D}(U) \backslash \{0\}$. Then,  $E^B_{d,\chi_1}(\Lambda) = E^B_{d,\chi_2}(\Lambda)$ with equivalent norms.}

Choose $\widetilde{\chi}_1, \widetilde{\chi}_2 \in \mathcal{D}(U)$ as in STEP II. Then,
$$
E^B_{d,\chi_1}(\Lambda) = E^B_{d, \widetilde{\chi}_1}(\Lambda) =  E^B_{d, \widetilde{\chi}_2}(\Lambda)  = E^B_{d,\chi_2}(\Lambda)
$$
with equivalent norms, where the first and third equality follow from STEP II and the second equality follows from STEP I.
\end{proof}

\begin{remark}\label{rem-ind-ope}
An immediate consequence of Theorem \ref{basic-sequence} is that $E^B_d(\Lambda)$ also does not depend on the bounded open set $U$ as long as the family of sets $\{\lambda +U\,|\, \lambda\in\Lambda\}$ is pairwise disjoint.
\end{remark}

The next result, which will be used later on, follows from an inspection of the proof of Theorem \ref{basic-sequence}.
\begin{lemma}\label{bound-bel}
Let $A \subset \DD(U)\backslash\{0\}$ be a bounded subset of $\DD(U)$.
\begin{itemize}
\item[$(i)$] For every $\chi \in \DD(U) \backslash\{0\}$ there is $C >0$ such that
   $$
   \sup_{\varphi \in A} \| c \|_{E^B_{d,\varphi}(\Lambda)} \leq C \| c \|_{E^B_{d,\chi}(\Lambda)}, \qquad c \in E^B_{d}(\Lambda).
   $$
\item[$(ii)$] Suppose that there is a non-empty open subset $V$ of $U$ such that
    \beqs
    \inf_{\varphi \in A} \inf_{x \in V} |\varphi(x)| > 0.
    \eeqs
    Then,  for every $\chi\in \DD(U)\backslash\{0\}$ there is $C > 0 $ such that
   $$
  \| c \|_{E^B_{d,\chi}(\Lambda)}\leq C  \inf_{\varphi \in A} \| c \|_{E^B_{d,\varphi}(\Lambda)} , \qquad c \in E^B_{d}(\Lambda).
   $$
\end{itemize}
\end{lemma}
Consider the following discrete spaces
$$
\mathcal{S}_{d}(\Lambda) := \{ c \in \C^\Lambda \, | \, \| c\|_{\mathcal{S}^N_{d}(\Lambda)} := \sup_{\lambda \in \Lambda} |c_\lambda|  (1+|\lambda|)^N < \infty, \, \forall N \in \N\}
$$
and
$$
\mathcal{S}'_{d}(\Lambda) := \{ c  \in \C^\Lambda \, | \, \exists N \in \N \, : \, \| c\|_{\mathcal{S}^{-N}_{d}(\Lambda)} := \sup_{\lambda \in \Lambda} |c_\lambda| (1+|\lambda|)^{-N} < \infty \},
$$
and endow them with their natural Fr\'echet space  and $(LB)$-space topology, respectively. The strong dual of $\mathcal{S}_{d}(\Lambda)$ may be topologically identified with $\mathcal{S}'_{d}(\Lambda)$. We then have:
\begin{proposition} \label{embedding-sequences} The following continuous inclusions hold
$$
\mathcal{S}_{d}(\Lambda) \rightarrow E^B_d(\Lambda) \rightarrow \mathcal{S}'_{d}(\Lambda).
$$
\end{proposition}
In view of the continuous inclusions $\mathcal{S}(\R^n) \rightarrow E \rightarrow \mathcal{S}'(\R^n)$,
Proposition \ref{embedding-sequences} is a direct consequence of the next lemma.
\begin{lemma}\label{sequence-mother} Let $\varphi \in \mathcal{S}(\R^n)$.
\begin{itemize}
\item[$(i)$] The mapping
$$
 \mathcal{S}_{d}(\Lambda) \rightarrow \mathcal{S}(\R^n), \, c \mapsto \sum_{\lambda \in \Lambda} c_\lambda T^B_{\lambda}\varphi
$$
is well-defined and continuous, and the series $\sum_{\lambda \in \Lambda} c_\lambda T^B_{\lambda}\varphi$ is absolutely summable in $\mathcal{S}(\R^n)$.
\item[$(ii)$] The mapping
\begin{equation}
 \mathcal{S}'_{d}(\Lambda) \rightarrow \mathcal{S}'(\R^n), \, c \mapsto \sum_{\lambda \in \Lambda} c_\lambda T^B_{\lambda}\varphi
\label{sequence-spaces-test}
\end{equation}
is well-defined and continuous, and the series $\sum_{\lambda \in \Lambda} c_\lambda T^B_{\lambda}\varphi$ is absolutely summable in $\mathcal{S}'(\R^n)$.
\item[$(iii)$] Suppose that $\varphi \in \mathcal{D}(U) \backslash \{0\}$. Then, $c \in \C^\Lambda$ belongs to  $\mathcal{S}'_{d}(\Lambda)$ if and only if  $ \sum_{\lambda \in \Lambda} c_\lambda T^B_{\lambda}\varphi \in  \mathcal{S}'(\R^n)$. Moreover, the mapping in \eqref{sequence-spaces-test} is a topological embedding.
\end{itemize}
\end{lemma}
\begin{proof} Parts $(i)$ and $(ii)$ are easy consequences of Lemma \ref{twisted-test} and we omit their proofs. We now show $(iii)$. Let $c \in \C^\Lambda$ be such that $\sum_{\lambda \in \Lambda} c_\lambda T^B_{\lambda}\varphi \in  \mathcal{S}'(\R^n)$. Hence, there are $N \in \N$ and $C > 0$ such that
$$
 \left| \left\langle \sum_{\lambda \in \Lambda} c_\lambda T^B_\lambda \varphi,  \psi \right\rangle \right| \leq C \| \psi \|_{\mathcal{S}^N}, \qquad \psi \in \mathcal{D}(\R^n).
$$
Pick $\psi \in \mathcal{D}(U)$ such that $\int_{\R^n} \varphi(x) \psi(x)\dx =1$. Then,
$$
 \left\langle \sum_{\lambda' \in \Lambda} c_{\lambda'} T^B_{\lambda'} \varphi, T^{-B}_{\lambda} \psi \right\rangle = c_{\lambda} \int_{\R^n} T^B_{\lambda} \varphi(x) T^{-B}_{\lambda}\psi(x) \dx = c_{\lambda}, \qquad \lambda \in \Lambda.
$$
 Lemma \ref{twisted-test}$(i)$ now implies that, for all $\lambda \in \Lambda$,
 \begin{align*}
|c_{\lambda}| =   \left| \left\langle \sum_{\lambda' \in \Lambda} c_{\lambda'} T^B_{\lambda'} \varphi, T^{-B}_{\lambda} \psi \right\rangle\right| \leq C \|T^{-B}_{\lambda} \psi \|_{ \mathcal{S}^N} \leq C (1+2\pi \|B\|)^N \| \psi \|_{\mathcal{S}^N}  (1+|{\lambda}|)^{2N},
\end{align*}
whence $c \in \mathcal{S}'_{d}(\Lambda)$. Finally, we show that the continuous mapping \eqref{sequence-spaces-test} is a topological embedding. It is clear that this mapping is injective and, by what we have just shown, it also has closed range. Since $\mathcal{S}'(\R^n)$ is a $(DFS)$-space  and  a closed subspace of a $(DFS)$-space is again a $(DFS)$-space, we obtain that the range of the mapping \eqref{sequence-spaces-test} is a $(DFS)$-space. Hence,  the result follows from the De Wilde open mapping theorem \cite[Theorem 1, p. 59]{kothe2} (cf.\ \cite[Theorem 8, p. 63]{kothe2}).
\end{proof}
Next, we give two results that will play a crucial role in the rest of the article.
The following result is the analogue of \cite[Proposition 5.2]{F-G1}  in our setting (see also the proof of \cite[Theorem 12.2.4]{Grochenig}).
\begin{proposition}\label{lemma-synthesis}
The bilinear mapping
$$
E^B_d(\Lambda)\times \SSS(\R^n) \rightarrow E, \, (c,\varphi) \mapsto S_{\varphi}(c),\quad \mbox{with}\quad S_{\varphi}(c)= \sum_{\lambda \in \Lambda} c_\lambda T^B_\lambda \varphi,
$$
is well-defined and continuous and uniquely extends to a continuous bilinear mapping
\begin{equation}\label{bil-synthesis-mappiing}
E^B_d(\Lambda)\times W(\mathcal{F}L^1_{\widetilde{\nu}_E},L^1_{\omega_E}) \rightarrow E,\, (c,\varphi)\mapsto\widetilde{S}_{\varphi}(c).
\end{equation}
Furthermore, if $E$ is a DTMIB with $E = E_0'$, where $E_0$ is a TMIB, there is $\chi \in \mathcal{D}(U)\backslash\{0\}$ such that for every $g \in E_0$ and $\varphi\in W(\mathcal{F}L^1_{\widetilde{\nu}_E},L^1_{\omega_E})$ there is $h \in E_0$ such that
\begin{equation}\label{change-for-element-dual}
\langle \widetilde{S}_{\varphi}(c), g \rangle= \left\langle S_\chi(c), h \right\rangle,\qquad  c\in E^B_d(\Lambda).
\end{equation}
\end{proposition}

\begin{proof} Let $r > 0$ be such that $[-4r,4r]^n \subset U$ and let $\chi \in \mathcal{D}_{[-r,r]^n}$ be such that $\sum_{m \in \Z^n} T_{rm} \chi = 1$ on $\R^n$. Choose $\psi\in \DD_{[-2r,2r]^n}$ such that $\psi = 1$ on $[-r,r]^n$ and $\psi_1 \in \mathcal{D}_{[-3r,3r]^n}$ such that $\psi_1=1$ on $[-2r,2r]^n$. Let $c \in E^B_d(\Lambda)$ and $\varphi\in W(\mathcal{F}L^1_{\widetilde{\nu}_E},L^1_{\omega_E})$ be arbitrary. For each $m \in \Z^n$, we infer
\begin{align}
\label{periodic-formula1}
\sum_{\lambda \in \Lambda} c_\lambda T^B_\lambda (\varphi T_{rm}\chi) &= \sum_{\lambda \in \Lambda}  c_\lambda e^{-2\pi i B\lambda\cdot rm} T_{rm}T^B_\lambda ((T_{-rm}\varphi)\chi ) \\ \nonumber
&=T_{rm}\sum_{\lambda \in \Lambda}  c_\lambda e^{-2\pi i B^trm\cdot\lambda} T^B_\lambda(\chi  (T_{-rm}\varphi) \psi) \\ \nonumber
&= T_{rm}\left(\sum_{\lambda \in \Lambda}  c_\lambda T^B_\lambda \chi \cdot \sum_{{\lambda'} \in \Lambda} e^{-2\pi i B^trm\cdot\lambda'} T_{\lambda'}( (\psi T_{-rm}\varphi) \psi_1 )\right).
\end{align}
Hence, Corollary \ref{mult-new} yields that $\sum_{\lambda \in \Lambda} c_\lambda T^B_\lambda (\varphi T_{rm}\chi) \in E$ and that
\begin{align*}
\sum_{m \in \Z^n} \left\|\sum_{\lambda \in \Lambda} c_\lambda T^B_\lambda (\varphi T_{rm}\chi)\right\|_E& \leq C\|c\|_{E^B_d(\Lambda)}\sum_{m \in \Z^n}\omega_E(rm) \|\psi T_{-rm}\varphi\|_{\mathcal{F}L^1_{\widetilde{\nu}_E}}\\
&= C\|c\|_{E^B_d(\Lambda)}\sum_{m \in \Z^n}\omega_E(rm)\|\varphi T_{rm}\psi\|_{\mathcal{F}L^1_{\widetilde{\nu}_E}}.
\end{align*}
Choose $\psi_2 \in \DD_{[-4r,4r]^n}$ such that $\psi_2 = 1$ on $[-3r,3r]^n$. Then,
\begin{align*}
\sum_{m \in \Z^n}\omega_E(rm)\|\varphi T_{rm}\psi\|_{\mathcal{F}L^1_{\widetilde{\nu}_E}}  &= r^{-n} \sum_{m \in \Z^n} \int_{rm + [-r/2,r/2]^n} \|\varphi T_{rm}\psi\|_{\mathcal{F}L^1_{\widetilde{\nu}_E}}  \omega_E(rm) \dx \\
& \leq C'r^{-n}  \sum_{m \in \Z^n} \int_{rm + [-r/2,r/2]^n} \|\varphi T_x\psi_2T_{rm}\psi\|_{\mathcal{F}L^1_{\widetilde{\nu}_E}}  \omega_E(x) \dx \\
& \leq C'r^{-n} \|\psi\|_{\mathcal{F}L^1_{\widetilde{\nu}_E}} \sum_{m \in \Z^n} \int_{rm + [-r/2,r/2]^n} \|\varphi T_x\psi_2\|_{\mathcal{F}L^1_{\widetilde{\nu}_E}}\omega_E(x) \dx \\
&= C'r^{-n}  \|\psi\|_{\mathcal{F}L^1_{\widetilde{\nu}_E}} \|\varphi\|_{W(\mathcal{F}L^1_{\widetilde{\nu}_E}, L^1_{\omega_E})}.
\end{align*}
We deduce that
\begin{equation}
\label{est-for-partt1}
\sum_{m \in \Z^n} \left\|\sum_{\lambda \in \Lambda} c_\lambda T^B_\lambda (\varphi T_{rm}\chi)\right\|_E \leq C''\|c\|_{E^B_d(\Lambda)} \|\varphi\|_{W(\mathcal{F}L^1_{\widetilde{\nu}_E}, L^1_{\omega_E})}.
\end{equation}
Now suppose that $\varphi \in \mathcal{S}(\R^n)$. Since the double series
$$
\sum_{m \in \Z^n,\, \lambda \in \Lambda} c_\lambda T^B_\lambda (\varphi T_{rm}\chi)
$$
is absolutely summable in $\SSS'(\R^n)$, we have that (cf.\ Lemma \ref{sequence-mother}$(ii)$)
\begin{equation}
\label{periodic-formula1-2}
\sum_{\lambda \in \Lambda} c_\lambda T^B_\lambda \varphi = \sum_{m \in \Z^n} \sum_{\lambda \in \Lambda} c_\lambda T^B_\lambda (\varphi T_{rm}\chi) \quad \mbox{in } \SSS'(\R^n),\quad c\in E^B_d(\Lambda), \varphi\in\SSS(\R^n).
\end{equation}
As $\mathcal{S}(\R^n)$ is dense in $W(\mathcal{F}L^1_{\widetilde{\nu}_E},L^1_{\omega_E})$, the first statement is therefore a consequence of \eqref{est-for-partt1}. Moreover, we obtain that
\begin{equation}\label{the-bilinear-synthesis}
\widetilde{S}_{\varphi}(c)= \sum_{m \in \Z^n} \sum_{\lambda \in \Lambda} c_\lambda T^B_\lambda (\varphi T_{rm}\chi),\quad c\in E^B_d(\Lambda),\, \varphi\in W(\mathcal{F}L^1_{\widetilde{\nu}_E},L^1_{\omega_E}).
\end{equation}
Next, suppose that $E$ is a DTMIB with $E = E_0'$, where $E_0$ is a TMIB. Let $g \in E_0$ and $\varphi\in W(\mathcal{F}L^1_{\widetilde{\nu}_E}, L^1_{\omega_E})$ be arbitrary. Similarly as in the proof of \eqref{est-for-partt1}, one can show that the series
$$
\sum_{m \in \Z^n} \left(T_{-rm}  g\cdot \sum_{{\lambda'} \in \Lambda} e^{-2\pi i B^trm\cdot \lambda'} T_{\lambda'}(\psi T_{-rm}\varphi)\right)
$$
is absolutely summable in $E_0$; denote it by $h\in E_0$. Then, \eqref{periodic-formula1} and \eqref{the-bilinear-synthesis} give \eqref{change-for-element-dual}.
\end{proof}

\begin{corollary} \label{inclusion-amalgam-E}
The space $W(\mathcal{F}L^1_{\widetilde{\nu}_E},L^1_{\omega_E})$ is continuously included into $E$. Consequently, $W(\mathcal{F}L^1_{\widetilde{\nu}_E},L^1_{\check{\omega}_E}) \subset E'$ continuously  if $E$ is a TMIB and $W(\mathcal{F}L^1_{\widetilde{\nu}_E},L^1_{\check{\omega}_{E}}) \subset E_0$ continuously  if $E$ is a DTMIB and $E = E'_0$, where $E_0$ is a TMIB.
\end{corollary}
\begin{proof}
Let $c^0 \in \C^\Lambda$ be such that $c^0_0 = 1$ and $c^0_\lambda = 0$ for $\lambda \in \Lambda \backslash \{0\}$. Since $\mathcal{S}(\R^n)$ is dense in $W(\mathcal{F}L^1_{\widetilde{\nu}_E},L^1_{\omega_E})$, Proposition \ref{lemma-synthesis} yields that $\widetilde{S}_{\varphi}(c^0) = \varphi$ for all $\varphi \in W(\mathcal{F}L^1_{\widetilde{\nu}_E},L^1_{\omega_E})$. The result now follows from another application of Proposition \ref{lemma-synthesis}.
\end{proof}
\begin{remark}
From now on, we will denote the continuous extension $\widetilde{S}_{\varphi}(c)$ simply by $S_{\varphi}(c)$. We emphasise that, at the moment, we do not claim that $S_{\varphi}(c)$ is given by $\sum_{\lambda \in \Lambda} c_\lambda T^B_\lambda \varphi$ for general $\varphi\in W(\mathcal{F}L^1_{\widetilde{\nu}_E},L^1_{\omega_E})$ as we do not give any meaning to this series for such $\varphi$. Later on, we will prove that  the series $\sum_{\lambda\in\Lambda} c_{\lambda} T^B_{\lambda}\varphi$  converges to $S_{\varphi}(c)$  in the C\'esaro sense (see Corollary \ref{equ-for-bilinear-mapping-sum} below).
\end{remark}

We now show a sampling inequality for the twisted translation; it should be compared with \cite[Lemma 3.9$(a)$ and Proposition 5.2]{Grochenig} and \cite[Proposition 11.1.4]{Grochenig}.

Corollary \ref{inclusion-amalgam-E} implies that the  bilinear mapping
\begin{equation}
\ast_B :E \times (\theta_{-B}W(\mathcal{F}L^1_{\widetilde{\nu}_E},L^1_{\check{\omega}_E}) )\check{}\rightarrow C_{1/\check{\rho}^{B}_{E}}
\label{amalgam-twisted-conv}
\end{equation}
is well-defined and continuous (cf. the last part of Section \ref{sect-twisted}).

\begin{proposition}\label{sampling}
The bilinear mapping
$$
E\times (\theta_{-B}W(\mathcal{F}L^1_{\widetilde{\nu}_E},L^1_{\check{\omega}_E}))\check{}\rightarrow E^B_d(\Lambda), \, (e,\varphi) \mapsto R_{\varphi}(e) := (e \ast_B \varphi(\lambda))_{\lambda \in \Lambda},
$$
is well-defined and continuous.
\end{proposition}

The proof of Proposition \ref{sampling} is based on the identity shown in the next lemma.
\begin{lemma}\label{identity-for-sampling}
For all $f\in \SSS'(\R^n)$, $\varphi\in\SSS(\R^n)$ and $\chi\in\DD(U)$, it holds that
\begin{equation}
\sum_{\lambda \in \Lambda} f \ast_B \varphi(\lambda) T^B_\lambda \chi  = \int_{\R^n} T_x f \cdot \sum_{\lambda \in \Lambda} e^{-2\pi i B^tx\cdot \lambda} T_{\lambda}(T_x(\theta_B(\check{\varphi})) \chi) \dx\quad \mbox{in } \SSS'(\R^n),
\label{identity-sampling1}
\end{equation}
where the integral should be interpreted as an $\mathcal{S}'(\R^n)$-valued Pettis integral with respect to the weak-$\ast$ topology on $\mathcal{S}'(\R^n)$.
\end{lemma}

\begin{proof}  Note that the mapping
$$
\R^n \rightarrow  \DD_{L^{\infty}_{(1+|\cdot|)^{-1}}}(\R^n), \, x \mapsto  \sum_{\lambda \in \Lambda} e^{-2\pi i B^tx\cdot\lambda} T_{\lambda}(T_x(\theta_B(\check{\varphi})) \chi)
$$
is continuous. This implies that the mapping
$$
\R^n \rightarrow \mathcal{S}'(\R^n), \, x \mapsto T_x f \cdot \sum_{\lambda \in \Lambda} e^{-2\pi i B^tx\cdot\lambda} T_{\lambda}(T_x(\theta_B(\check{\varphi})) \chi)
$$
is continuous with respect to the weak-$\ast$ topology on $\mathcal{S}'(\R^n)$. Hence, by Lemma \ref{sequence-mother}$(ii)$, we only need to show that
$$
 \sum_{\lambda \in \Lambda} \int_{\R^n}  f \ast_B \varphi(\lambda) T^B_\lambda\chi(x) \psi(x) \dx  = \int_{\R^n} \left\langle T_x f \cdot \sum_{\lambda \in \Lambda} e^{-2\pi i B^tx\cdot\lambda} T_{\lambda}(T_x(\theta_B(\check{\varphi})) \chi), \psi \right\rangle \dx
$$
for all $\psi \in \mathcal{S}(\R^n)$. We have that
\begin{align*}
&\sum_{\lambda \in \Lambda} \int_{\R^n}  f \ast_B \varphi(\lambda) T^B_\lambda \chi(x) \psi(x) \dx \\
&= \sum_{\lambda \in \Lambda} \langle f(t), T^{-B}_\lambda (\theta_B(\check{\varphi}))(t) \rangle  \int_{\R^n}T^B_\lambda \chi(x) \psi(x) \dx \\
&=\sum_{\lambda \in \Lambda} \left\langle f(t), \int_{\R^n}T^{-B}_\lambda (\theta_B(\check{\varphi}))(t) T^B_\lambda \chi(t+x) \psi(t+x) \dx \right\rangle.
\end{align*}
As the function
$$
\R^{2n}\rightarrow \C,\, (t,x)\mapsto T^{-B}_\lambda (\theta_B(\check{\varphi}))(t) T^B_\lambda \chi(t+x) \psi(t+x),
$$
belongs to $\SSS(\R^{2n})$, we infer that
\begin{align*}
&\sum_{\lambda \in \Lambda}\int_{\R^n} f \ast_B \varphi(\lambda) T^B_\lambda \chi(x) \psi(x) \dx \\
&= \sum_{\lambda \in \Lambda} \langle f(t) \otimes 1(x), T^{-B}_\lambda (\theta_B(\check{\varphi}))(t) T^B_\lambda \chi(t+x) \psi(t+x)  \rangle \\
&= \sum_{\lambda \in \Lambda} \int_{\R^n} \langle f(t), T^{-B}_\lambda (\theta_B(\check{\varphi}))(t) T^B_\lambda \chi(t+x) \psi(t+x) \rangle \dx \\
&= \int_{\R^n} \sum_{\lambda \in \Lambda} \langle f, T^{-B}_\lambda (\theta_B(\check{\varphi})) T_{-x} (T^B_\lambda \chi \psi) \rangle \dx \\
&= \int_{\R^n}  \sum_{\lambda \in \Lambda} \langle T_xf,   T_xT^{-B}_\lambda(\theta_B(\check{\varphi}))T^B_\lambda \chi \psi \rangle \dx\\
&= \int_{\R^n} \sum_{\lambda \in \Lambda}  \langle T_xf, e^{-2\pi i B^tx\cdot \lambda} T_{\lambda}(T_x(\theta_B(\check{\varphi})) \chi) \psi \rangle \dx \\
&= \int_{\R^n} \left\langle T_xf \cdot \sum_{\lambda \in \Lambda} e^{-2\pi i B^tx\cdot \lambda}T_{\lambda}(T_x(\theta_B(\check{\varphi})) \chi),  \psi \right\rangle \dx.
\end{align*}
This completes the proof of the lemma.
\end{proof}
\begin{proof}[Proof of Proposition \ref{sampling}] As $e\ast_B\varphi$ is continuous, we can evaluate it at $\lambda\in\Lambda$. Fix $\chi \in \mathcal{D}(U) \backslash \{0\}$. Since $\mathcal{S}(\R^n)$ is dense in $W(\mathcal{F}L^1_{\widetilde{\nu}_E},L^1_{\check{\omega}_E})$, Lemma \ref{identity-for-sampling} and the continuity of the mapping \eqref{amalgam-twisted-conv} imply that it suffices to show that the bilinear mapping
\begin{gather}
E\times (\theta_{-B}W(\mathcal{F}L^1_{\widetilde{\nu}_E},L^1_{\check{\omega}_E}))\check{}\rightarrow E,\nonumber \\
 \, (e,\varphi) \mapsto  \int_{\R^n} T_x e \cdot \sum_{\lambda \in \Lambda} e^{-2\pi i B^tx\cdot \lambda} T_{\lambda}(T_x(\theta_B(\check{\varphi})) \chi) \dx
\label{mapping-proof-sampling11}
\end{gather}
is well-defined and continuous, where  the integral should be interpreted as an $E$-valued Bochner integral if $E$ is a TMIB and as an $E$-valued Pettis integral if $E$ is a DTMIB.
Let  $e \in E$ and $\varphi \in (\theta_{-B}W(\mathcal{F}L^1_{\widetilde{\nu}_E},L^1_{\check{\omega}_E}))\check{}$ be arbitrary. Choose $\chi_1\in\DD(U)$ such that $\chi_1=1$ on $\operatorname{supp}\chi$. Then,
$$
\sum_{\lambda \in \Lambda} e^{-2\pi i B^tx\cdot \lambda} T_{\lambda}(T_x(\theta_B(\check{\varphi})) \chi) =   \sum_{\lambda \in \Lambda} e^{-2\pi i B^tx\cdot \lambda} T_{\lambda}(T_x(\theta_B(\check{\varphi})) \chi \chi_1), \qquad x \in \R^n.
$$
Hence, Corollary \ref{mult-new} verifies that, for $x \in \R^n$ fixed, the integrand in \eqref{mapping-proof-sampling11} is a well-defined element of $E$.  \\
\\
\noindent \textbf{Claim.} The mapping
\begin{equation}\label{the-mapping-for-the-integral-for-meas}
\R^n \rightarrow E,\, x\mapsto T_xe\cdot \sum_{\lambda\in\Lambda} e^{-2\pi i B^tx\cdot\lambda} T_{\lambda}(T_x(\theta_B(\check{\varphi}))\chi),
\end{equation}
is strongly measurable if $E$ is a TMIB and weak-$\ast$ measurable if $E$ is a DTMIB.\\
\\
\noindent Assuming the validity of the claim, Corollary \ref{mult-new} gives the bound
\begin{align*}
&\int_{\R^n} \left\|T_x e \cdot \sum_{\lambda \in \Lambda} e^{-2\pi i B^tx\cdot \lambda} T_{\lambda}((T_x(\theta_B(\check{\varphi})) \chi)\chi_1)\right\|_E \dx\\
&\leq C \|e\|_E \int_{\R^n} \omega_E(x) \|T_x(\theta_B(\check{\varphi})) \chi \|_{\mathcal{F}L^1_{\widetilde{\nu}_E}} \dx=   C \|e\|_E \| \varphi\|_{(\theta_{-B}W(\mathcal{F}L^1_{\widetilde{\nu}_E},L^1_{\check{\omega}_E}))\check{}},
\end{align*}
whence the mapping \eqref{mapping-proof-sampling11} is well-defined and continuous. It remains to prove the claim. First suppose that $\varphi\in\SSS(\R^n)$. For each $x\in\R^n$, it holds that
$$
\sum_{\lambda\in\Lambda} e^{-2\pi i B^tx\cdot\lambda} T_{\lambda}(T_x(\theta_B(\check{\varphi}))\chi)\in \DD_{L^{\infty}}(\R^n)\quad \mbox{and}\quad \sum_{\lambda\in\Lambda} T_{\lambda}(T^{-B^t}_x(\theta_B(\check{\varphi}))\chi)\in \DD_{L^{\infty}}(\R^n).
$$
We infer that, for all $\psi\in\SSS(\R^n)$ and $x\in\R^n$,
\begin{align*}
&\left\langle T_xe\cdot \sum_{\lambda\in\Lambda} e^{-2\pi i B^tx\cdot\lambda} T_{\lambda}(T_x(\theta_B(\check{\varphi}))\chi),\psi\right\rangle\\
&=\left\langle T_xe,\psi\sum_{\lambda\in\Lambda} e^{-2\pi i B^tx\cdot\lambda} T_{\lambda}(T_x(\theta_B(\check{\varphi}))\chi)\right\rangle\\
&=\left\langle T_xe,e^{2\pi i B^tx\cdot x} M_{-B^tx}\left(\psi\sum_{\lambda\in\Lambda} T_{\lambda}(T^{-B^t}_x(\theta_B(\check{\varphi}))\chi)\right)\right\rangle\\
&=\left\langle T^{B^t}_xe\cdot \sum_{\lambda\in\Lambda} T_{\lambda}(T^{-B^t}_x(\theta_B(\check{\varphi}))\chi),\psi\right\rangle,
\end{align*}
and, consequently,
\begin{equation}
\label{identity}
T_xe\cdot \sum_{\lambda\in\Lambda} e^{-2\pi i B^tx\cdot\lambda} T_{\lambda}(T_x(\theta_B(\check{\varphi}))\chi)=T^{B^t}_xe\cdot \sum_{\lambda\in\Lambda} T_{\lambda}(T^{-B^t}_x(\theta_B(\check{\varphi}))\chi).
\end{equation}
Lemma \ref{mea-spa-sub} implies that
$$
\sum_{\lambda\in\Lambda} T_{\lambda}(T^{-B^t}_x(\theta_B(\check{\varphi}))\chi) = \sum_{\lambda\in\Lambda} T_{\lambda}(T^{-B^t}_x(\theta_B(\check{\varphi}))\chi \chi_1)\in \mathcal{F}\mathcal{M}^1_{\widetilde{\nu}_E}
$$
and therefore the multiplication on the right-hand side in \eqref{identity} may be interpreted as the multiplication on $E\times \mathcal{F}\mathcal{M}^1_{\widetilde{\nu}_E}$.
Since the mapping $\R^n\rightarrow \mathcal{F}L^1_{\widetilde{\nu}_E}$, $x\mapsto T^{-B^t}_x(\theta_B(\check{\varphi}))$, is continuous, Lemma \ref{mea-spa-sub} gives the continuity of the mapping
$$
\R^n\rightarrow \mathcal{F}\mathcal{M}^1_{\widetilde{\nu}_E},\, x\mapsto \sum_{\lambda\in\Lambda} T_{\lambda}(T^{-B^t}_x(\theta_B(\check{\varphi}))\chi),
$$
which, in turn, yields that the mapping
$$
\R^n\rightarrow E,\, x\mapsto T^{B^t}_xe\cdot \sum_{\lambda\in\Lambda} T_{\lambda}(T^{-B^t}_x(\theta_B(\check{\varphi}))\chi),
$$
is continuous if $E$ is a TMIB and continuous with respect to the weak-$\ast$ topology on $E$ if $E$ is a DTMIB. Thus, if $\varphi\in\SSS(\R^n)$, the mapping \eqref{the-mapping-for-the-integral-for-meas} is continuous if $E$ is a TMIB and continuous with respect to the weak-$\ast$ topology on $E$ if $E$ is a DTMIB. Now let $\varphi\in (\theta_{-B}W(\mathcal{F}L^1_{\widetilde{\nu}_E}, L^1_{\check{\omega}_E}))\check{}$ be arbitrary. Choose a sequence $(\varphi_j)_{j\in\N}\subset\SSS(\R^n)$ that converges to $\varphi$ in $(\theta_{-B}W(\mathcal{F}L^1_{\widetilde{\nu}_E}, L^1_{\check{\omega}_E}))\check{}$.  Since the mapping $W(\mathcal{F}L^1_{\widetilde{\nu}_E}, L^1_{\check{\omega}_E})\rightarrow \mathcal{F}L^1_{\widetilde{\nu}_E}$, $\psi\mapsto \psi\chi$, is continuous, Corollary \ref{mult-new} implies that the mappings
$$
\R^n\rightarrow E,\, x\mapsto T_xe\cdot \sum_{\lambda\in\Lambda} e^{-2\pi i B^tx\cdot\lambda} T_{\lambda}(T_x(\theta_B(\check{\varphi}_j))\chi\chi_1),\quad j\in\N,
$$
converge pointwise to the mapping \eqref{the-mapping-for-the-integral-for-meas} in $E$ if $E$ is a TMIB and in the weak-$\ast$ topology of $E$ if $E$ is a DTMIB. This implies the claim.
\end{proof}

\begin{corollary}\label{complemented}
Let $\chi, \psi \in \mathcal{D}(U) \backslash \{0\}$ be such that $(\theta_B(\psi), \overline{\chi})_{L^2} \neq 0$. Then, the mappings
$$
S_{\chi} : E^B_d(\Lambda) \rightarrow  E \qquad \mbox{and} \qquad R_{\check{\psi}}: E \rightarrow  E^B_d(\Lambda)
$$
are continuous and
\begin{equation}
R_{\check{\psi}} \circ S_{\chi} = (\theta_B(\psi), \overline{\chi})_{L^2} \operatorname{id}_{E^B_d(\Lambda)}.
\label{complemented-seq}
\end{equation}
In particular, $S_{\chi}(E^B_d(\Lambda))$ is a complemented subspace of $E$.
\end{corollary}
\begin{proof}
The mapping $S_{\chi} : E^B_d(\Lambda) \rightarrow  E$ is continuous by definition of $E^B_d(\Lambda)$ and the continuity of the mapping $R_{\check{\psi}}: E \rightarrow  E^B_d(\Lambda)$ has been shown in Proposition \ref{sampling}. The identity \eqref{complemented-seq} follows from a straightforward computation.
\end{proof}
We end this subsection by giving two examples; further examples shall be discussed in Subsection \ref{sect-examples} below.
\begin{examples} \label{examples-sequences} $(i)$ Let $E$ be a solid TMIB or DTMIB. Fix a bounded open neighbourhood of the origin $W$ with $\overline{W}\subset U$. We define the Banach space
$$
E_d(\Lambda) := \left\{  c  \in \C^\Lambda  \, \Big| \,\sum_{\lambda \in \Lambda} c_\lambda  1_{\lambda+W} \in E \right\}
$$
with norm
$
\| c \| :=  \| \sum_{\lambda \in \Lambda} |c_\lambda|  1_{\lambda+W} \|_E.
$
Note that $E_d(\Lambda)$ is solid. We have that $E^B_d(\Lambda) = E_d(\Lambda)$ topologically for all real-valued $n \times n$-matrices $B$. Hence, in the solid case, our definition coincides with the standard one (cf. \cite[Definition 3.4]{F-G1}). Let $w$ be a polynomially bounded weight function on $\R^n$. For $1 \leq p \leq \infty$ we define $\ell^p_w(\Lambda)$ as the Banach space consisting of all $c \in \C^\Lambda$ such that $ \|c\|_{\ell^p_w(\Lambda)} := \| (c_\lambda w(\lambda))_{\lambda \in \Lambda} \|_{\ell^p(\Lambda)} < \infty$. We define $c_{0,w}(\Lambda)$ as the closed subspace of $\ell^{\infty}_w(\Lambda)$ consisting of all $c\in\ell^{\infty}_w(\Lambda)$ satisfying the following property: For every $\varepsilon>0$ there is a finite subset $\Lambda^{(0)}$ of  $\Lambda$ such that $\sup_{\lambda\in\Lambda\backslash\Lambda^{(0)}}|c_{\lambda}| w(\lambda)\leq \varepsilon$. Then, $(L^p_w)_d(\Lambda) = \ell^p_w(\Lambda)$; furthermore, $(C_{0,w})^B_d(\Lambda)=c_{0,w}(\Lambda)$ for all real-valued $n\times n$-matrices $B$. A similar statement holds for the weighted mixed-norm spaces considered in Example \ref{examples-TMIB}$(i)$.\\ \\
\noindent $(ii)$ Let $E$ be a solid TMIB or DTMIB. We wish to determine  $(\mathcal{F}E)^0_d(\Lambda)$. We define $E(\R^n/ \Lambda^\perp)$ as the Banach space consisting of all $\Lambda^\perp$-periodic elements $f \in E_{\operatorname{loc}}$   with norm
$$
\| f \|_{E(\R^n/ \Lambda^\perp)} := \| f  1_{I_{\Lambda^\perp}}\|_E.
$$
Since $E \subset L^1_{\operatorname{loc}}(\R^n)$, we have that $E(\R^n/ \Lambda^\perp) \subset L^1(\R^n/ \Lambda^\perp)$. As customary, we define the Fourier coefficients of an element $f \in  L^1(\R^n/ \Lambda^\perp)$ as
$$
c_\lambda(f) =  \frac{1}{\operatorname*{vol}(\Lambda^\perp)} \int_{I_{\Lambda^\perp}} f(x) e^{-2\pi i \lambda \cdot x} \dx, \qquad \lambda \in \Lambda.
$$
We then have:
\begin{proposition}\label{Fourier-coeff-seq}
Let $E$ be a solid TMIB or DTMIB. Then,
\begin{equation}
E(\R^n/ \Lambda^\perp) \rightarrow (\mathcal{F}E)^0_d(\Lambda), \, f \mapsto (c_\lambda(f))_{\lambda \in \Lambda}
\label{fourier-series-repr}
\end{equation}
is a topological isomorphism.
\end{proposition}
\begin{proof} Let $f \in E(\R^n/ \Lambda^\perp)$ be arbitrary. Note that $(c_\lambda(f))_{\lambda \in \Lambda} \in \ell^\infty(\Lambda) \subset \mathcal{S}'_d(\Lambda)$. Hence,
$$
f(\xi) = \sum_{\lambda \in \Lambda} c_\lambda(f) e^{2\pi i \lambda \cdot \xi}\,\, \mbox{ in $\mathcal{S}'(\R^n)$}.
$$
Let $\chi \in \mathcal{D}(U) \backslash \{0\}$. By Lemma \ref{sequence-mother}$(ii)$, we infer that
\begin{align*}
\mathcal{F}^{-1}(S_\chi((c_\lambda(f))_{\lambda \in \Lambda}))(\xi) &= \sum_{\lambda \in \Lambda}   c_\lambda(f) \mathcal{F}^{-1}(T_\lambda \chi) (\xi)= \mathcal{F}^{-1}(\chi) (\xi)\sum_{\lambda \in \Lambda} c_\lambda(f) e^{2\pi i \lambda \cdot \xi}\\
& = \mathcal{F}^{-1}(\chi)(\xi) f(\xi) = \sum_{\mu \in \Lambda^\perp} \mathcal{F}^{-1}(\chi)(\xi)f(\xi) 1_{\mu + I_{\Lambda^\perp}}(\xi).
\end{align*}
For each $\mu \in \Lambda^\perp$ it holds that
\begin{align*}
\|\mathcal{F}^{-1}(\chi)f 1_{\mu + I_{\Lambda^\perp}} \|_E &\leq \omega_E(\mu) \| T_{-\mu}(\mathcal{F}^{-1}(\chi)) f 1_{I_{\Lambda^\perp}}\|_E \\
&\leq C_0 (1+|\mu|)^{\tau_0} \| T_{-\mu}\mathcal{F}^{-1}(\chi)\|_{L^{\infty}(I_{\Lambda^\perp})} \| f \|_{E(\R^n/ \Lambda^\perp)} \\
&\leq C(1+|\mu|)^{-n-1} \| f \|_{E(\R^n/ \Lambda^\perp)},
\end{align*}
whence the mapping \eqref{fourier-series-repr} is well-defined and continuous. This mapping is injective because for all  $f \in  L^1(\R^n/ \Lambda^\perp)$ it holds that $f = 0$ if and only if $c_\lambda(f) = 0$ for all $\lambda \in \Lambda$. Next, let $c \in  (\mathcal{F}E)^0_d(\Lambda)$ be arbitrary. By Proposition \ref{embedding-sequences}, $c \in \mathcal{S}'_d(\Lambda)$. Hence,
$
f(\xi) = \sum_{\lambda \in \Lambda} c_\lambda e^{2\pi i \lambda \cdot \xi}
$
is a well-defined $\Lambda^\perp$-periodic element of $\mathcal{S}'(\R^n)$. Choose $\varphi \in \SSS(\R^n)$ such that $\mathcal{F}^{-1}(\varphi)(\xi) \neq 0$ for all $\xi \in \R^n$. By Lemma \ref{sequence-mother}$(ii)$ and Proposition \ref{lemma-synthesis}, we have that
$$
E\ni\mathcal{F}^{-1}(S_{\varphi}(c)) = \sum_{\lambda \in \Lambda}   c_{\lambda} \mathcal{F}^{-1}(T_{\lambda} \varphi) = \mathcal{F}^{-1}(\varphi)f,
$$
which implies that $f \in E(\R^n/ \Lambda^\perp)$ and, in view of Proposition \ref{lemma-synthesis},
$$
\| f \|_{E(\R^n/ \Lambda^\perp)}  = \| f 1_{I_{\Lambda^\perp}} \|_E \leq \|1/\mathcal{F}^{-1}(\varphi)\|_{L^{\infty}(I_{\Lambda^\perp})} \|\mathcal{F}^{-1}(\varphi) f \|_E \leq C \| c\|_{\mathcal{F}E^0_d(\Lambda)}.
$$
Clearly, $c$ is equal to the image of $f$ under the mapping \eqref{fourier-series-repr}. Therefore, this mapping is surjective and its inverse is continuous.
\end{proof}

\begin{corollary}\label{cor-for-equ-wight-notwei-four-spa}
Let $1 \leq p \leq \infty$. Then, $(\mathcal{F}L^p_w)^0_d(\Lambda) = (\mathcal{F}L^p)^0_d(\Lambda)$ for all polynomially bounded weight functions $w$ on $\R^n$.
\end{corollary}

\end{examples}
\subsection{Convergence properties} In this subsection we address the following question: \emph{Let $\chi \in \mathcal{D}(U)\backslash \{0\}$ and $c \in E^B_d(\Lambda)$. In which sense does the series $\sum_{\lambda \in \Lambda} c_\lambda T^B_\lambda \chi$ converge in $E$?} When $E$ is solid, we can give a quick answer to this question (cf.\ \cite[Proposition 5.2]{F-G1}).

\begin{lemma} Let $E$ be solid and let $\chi\in \DD(U)\backslash\{0\}$. For each $c\in E^B_d(\Lambda)$ the series $\sum_{\lambda\in\Lambda} c_{\lambda}T^B_{\lambda}\chi$ converges unconditionally in $E$ if $E$ is a TMIB and converges unconditionally with respect to the weak-$\ast$ topology on $E$ if $E$ is a DTMIB.
\end{lemma}

\begin{proof} We only consider the case when $E$ is a TMIB as the case when $E$ is a DTMIB can be treated similarly. Let $\varepsilon>0$ be arbitrary. Pick $\psi\in\DD(\R^n)$ such that $\|S_\chi(c)-\psi\|_E\leq\varepsilon$. Let $\Lambda^{(0)}$ be a finite subset of $\Lambda$ such that $\operatorname{supp}\psi\cap (\bigcup_{\lambda\in\Lambda\backslash\Lambda^{(0)}}(\lambda+U))=\emptyset$. For any $\Lambda^{(0)}\subseteq\Lambda'\subset \Lambda$, $\Lambda'$ finite, denote by $g_{\Lambda'}$ the characteristic function of the set $\operatorname{supp}\psi\cup(\bigcup_{\lambda\in\Lambda'}(\lambda+U))$. Since $\sum_{\lambda\in\Lambda'}c_{\lambda} T^B_{\lambda}\chi=g_{\Lambda'}S_\chi(c)$ and $(1-g_{\Lambda'})\psi=0$, we infer that
$$
\left\| S_\chi(c)-\sum_{\lambda\in\Lambda'}c_{\lambda} T^B_{\lambda}\chi\right\|_E=\|(1-g_{\Lambda'})(S_\chi(c)-\psi)\|_E\leq \|S_\chi(c)-\psi\|_E\leq \varepsilon.
$$
\end{proof}

However, for general TMIB and DTMIB this question is far more subtle, as the following observation shows.

\begin{proposition} \label{negative} Let $\chi \in \mathcal{D}((-\frac{1}{2}, \frac{1}{2})) \backslash \{0\}$. For $1\leq p < \infty$, $p \neq 2$, there exists an element $c \in (\mathcal{F}L^p)^0_d(\Z)$ such that the series $\sum_{\lambda \in \Z} c_\lambda T_\lambda \chi$ is not unconditionally convergent  in  $\mathcal{F}L^p(\R)$. For $p = 1$ there even exists an element  $c \in( \mathcal{F}L^1)^0_d(\Z)$ such that the sequence of symmetric partial sums $\left(\sum_{|\lambda| \leq N} c_\lambda T_\lambda \chi \right)_{N \in \N}$ does not converge in $\mathcal{F}L^1(\R)$.
\end{proposition}
\begin{proof}
In view of Proposition \ref{Fourier-coeff-seq}, this is a consequence of the following two classical facts about Fourier series: For $1\leq p < \infty$, $p \neq 2$, there exists an element in $L^p(\R/\Z)$ whose Fourier series is not unconditionally convergent in $L^p(\R/\Z)$ \cite[Exercise 6.5]{Heil}; there exists an element in $L^1(\R/\Z)$ such that the sequence of symmetric partial sums of its Fourier series does not converge  in $L^1(\R/\Z)$ \cite[Example 4.1.4]{Grafakos}.
\end{proof}
We will now formulate a positive answer to the above question by using the concept of C\'esaro summability.
\begin{definition}\label{def-cesaro} Let $X$ be a Hausdorff topological vector space and let $(x_\lambda)_{\lambda \in \Lambda} \subset X$.  The series $\sum_{\lambda \in \Lambda} x_\lambda$ is said to be \emph{C\'esaro summable to $x \in X$} if  (recall that $\Lambda = A_\Lambda \Z^n$)
$$
\lim_{N \to \infty} \sum_{\substack{m \in \Z^n \\ |m_j| < N}} \left ( 1 - \frac{|m_1|}{N} \right) \cdots \left ( 1 - \frac{|m_n|}{N} \right) x_{A_\Lambda m} = x.
$$
\end{definition}

\begin{theorem} \label{cesaro} Let $\chi \in \mathcal{D}(U) \backslash \{0\}$. For each $c \in E^B_d(\Lambda)$, the series $\sum_{\lambda \in \Lambda} c_\lambda T^B_\lambda \chi$ is C\'esaro summable in $E$ if $E$ is a TMIB and C\'esaro summable with respect to the weak-$\ast$ topology on $E$ if $E$ is a DTMIB.
\end{theorem}

We need some preparation for the proof of Theorem \ref{cesaro}. A sequence $(k_N)_{N \in \Z_+} \subset L^1(\R^n / \Z^n)$ is called an \emph{approximate identity on $\R^n / \Z^n$} \cite[Definition 1.2.15]{Grafakos} if
\begin{itemize}
\item[$(i)$] $ \int_{[-\frac{1}{2},\frac{1}{2}]^n} k_N(x) \dx = 1$ for all $N \in\Z_+$.
\item[$(ii)$] $ \sup_{N \in \Z_+} \int_{[-\frac{1}{2},\frac{1}{2}]^n} |k_N(x)| \dx < \infty$.
\item[$(iii)$] For all $\delta > 0$ it holds that
$$
\lim_{N \to \infty}  \int_{[-\frac{1}{2},\frac{1}{2}]^n \backslash [-\delta,\delta]^n} |k_N(x)| \dx = 0.
$$
\end{itemize}
Set
$$
F_N(x) = \sum_{\substack{m \in \Z^n \\ |m_j| < N}} \left ( 1 - \frac{|m_1|}{N} \right) \cdots \left ( 1 - \frac{|m_n|}{N} \right) e^{2\pi i m \cdot x}, \qquad N \in \Z_+.
$$
Then, $(F_N)_{N \in \Z_+}$, called the \emph{F\'ejer kernel}, is an approximate identity on $\R^n / \Z^n$ \cite[Proposition 3.1.10]{Grafakos}. We need the following vector-valued version of the fundamental property of approximate identities on $\R^n / \Z^n$; the proof is analogous to the scalar-valued case (see e.g.\ the proof of \cite[Theorem 1.2.19]{Grafakos}) and we omit it.

\begin{lemma} \label{approx}  Let $(k_N)_{N \in \Z_+}$  be an approximate identity on $\R^n / \Z^n$.
\begin{itemize}
\item[$(i)$] Let $X$ be a Banach space. Suppose that ${\bf{f}}: [-\frac{1}{2},\frac{1}{2}]^n \rightarrow X$ is continuous. Then,
\begin{equation}
\lim_{N \to \infty}  \int_{[-\frac{1}{2},\frac{1}{2}]^n} {\bf{f}}(x)k_N(x) \dx = {\bf{f}}(0),
\label{approx-2}
\end{equation}
where the above integrals should be interpreted as $X$-valued Bochner integrals.
\item[$(ii)$] Let $X_0$ be a separable Banach space and set $X = X_0'$. Suppose that ${\bf{f}}: [-\frac{1}{2},\frac{1}{2}]^n \rightarrow X$ is continuous with respect to the weak-$\ast$ topology on $X$. Then, \eqref{approx-2} holds with respect to the weak-$\ast$ topology on $X$ if the integrals are interpreted as $X$-valued Pettis integrals with respect to the weak-$\ast$ topology on $X$.
\end{itemize}
\end{lemma}

\begin{proof}[Proof of Theorem \ref{cesaro}]
Choose $\psi \in \mathcal{D}(U)$ such that $\psi = 1$ on $\operatorname{supp} \chi$.  Let $c \in E^B_d(\Lambda)$ be arbitrary. For each $N \in \Z_+$, it holds that
\begin{align*}
&\sum_{|m_j| <N} \left ( 1 - \frac{|m_1|}{N} \right) \cdots \left ( 1 - \frac{|m_n|}{N} \right) c_{A_\Lambda m}  T^B_{A_\Lambda m} \chi \\
&= \sum_{|m_j| <N} \left ( 1 - \frac{|m_1|}{N} \right) \cdots \left ( 1 - \frac{|m_n|}{N} \right) c_{A_\Lambda m}  T^B_{A_\Lambda m} \chi T_ {A_\Lambda m} \psi \\
&= S_\chi(c) \cdot   \sum_{|m_j| <N} \left ( 1 - \frac{|m_1|}{N} \right) \cdots \left ( 1 - \frac{|m_n|}{N} \right) T_ {A_\Lambda m} \psi.
\end{align*}
Note that
\begin{align*}
\mathcal{F}^{-1} \left( \sum_{|m_j| <N}  \left( 1 - \frac{|m_1|}{N} \right) \cdots \left ( 1 - \frac{|m_n|}{N} \right) T_ {A_\Lambda m} \psi\right) (\xi)
 = \widehat{\psi}(-\xi) F_N(A^t_\Lambda \xi).
\end{align*}
Hence, \eqref{int-22} yields that
$$
\sum_{|m_j| <N} \left ( 1 - \frac{|m_1|}{N} \right) \cdots \left ( 1 - \frac{|m_n|}{N} \right) c_{A_\Lambda m}  T^B_{A_\Lambda m} \chi= \int_{\R^n}  M_{-\xi} S_\chi(c)  \widehat{\psi}(-\xi)  F_N(A^t_\Lambda \xi) \dxi,
$$
where the integral should be interpreted as an $E$-valued Bochner integral if $E$ is a TMIB and as an $E$-valued Pettis integral if $E$ is a DTMIB. Therefore, Lemma \ref{approx} yields that
\begin{align*}
& \lim_{N \to \infty} \sum_{|m_j| <N} \left ( 1 - \frac{|m_1|}{N} \right) \cdots \left ( 1 - \frac{|m_n|}{N} \right) c_{A_\Lambda m}  T^B_{A_\Lambda m} \chi \\
&= \frac{1}{\operatorname{vol}(\Lambda)} \lim_{N \to \infty} \int_{\R^n}  M_{(A^t_\Lambda)^{-1}\xi} S_\chi(c)   \widehat{\psi}((A^t_\Lambda)^{-1}\xi)F_N(\xi)  \dxi \\
&= \frac{1}{\operatorname{vol}(\Lambda)} \lim_{N \to \infty} \sum_{m \in \Z^n} \int_{m + [-\frac{1}{2},\frac{1}{2}]^n } M_{(A^t_\Lambda)^{-1}\xi} S_\chi(c)  \widehat{\psi}((A^t_\Lambda)^{-1}\xi)  F_N(\xi)  \dxi \\
&= \frac{1}{\operatorname{vol}(\Lambda)} \sum_{m \in \Z^n} \lim_{N \to \infty}  \int_{[-\frac{1}{2},\frac{1}{2}]^n }   M_{(A^t_\Lambda)^{-1}(m+\xi)} S_\chi(c)  \widehat{\psi}((A^t_\Lambda)^{-1}(m + \xi)) F_N(\xi) \dxi \\
&= \frac{1}{\operatorname{vol}(\Lambda)} \sum_{m \in \Z^n} M_{(A^t_\Lambda)^{-1}m} S_\chi(c) \widehat{\psi}((A^t_\Lambda)^{-1}m)\\
&= \frac{1}{\operatorname{vol}(\Lambda)} \sum_{\mu \in \Lambda^\perp} M_{\mu} S_\chi(c)\widehat{\psi}(\mu) \\
&= S_\chi(c) \cdot \sum_{\lambda \in \Lambda} T_\lambda \psi\\
&=  S_\chi(c),
\end{align*}
where the last equality follows from the fact that $\psi = 1$ on $\operatorname{supp} \chi$.
\end{proof}
\begin{remark}
Let $\chi \in \mathcal{D}(U) \backslash \{0\}$ and let $c \in E^B_d(\Lambda)$. Instead of the C\'esaro means, we can also consider the \emph{Bochner-Riesz means} of order $\alpha$, $\alpha \geq 0$, of the series $S_\chi(c) = \sum_{\lambda \in \Lambda} c_\lambda T^B_\lambda\chi$, namely,
$$
B^\alpha_N(S_\chi(c)) = \sum_{\substack{m \in \Z^n \\ |m| \leq N }}  \left ( 1 - \frac{|m|^2}{N^2} \right)^\alpha c_{A_{\Lambda} m} T^B_{A_\Lambda m} \chi, \qquad N \in \Z_+.
$$
Set
$$
L^\alpha_N(x) = \sum_{\substack{m \in \Z^n \\ |m| \leq N }}  \left ( 1 - \frac{|m|^2}{N^2} \right)^\alpha e^{2\pi i m \cdot x}, \qquad N \in \Z_+.
$$
Then, $(L^\alpha_N)_{N \in \Z_+}$ is an approximate identity on $\R^n / \Z^n$ if $\alpha > (n-1)/2$ \cite[Proof of Proposition 4.1.9]{Grafakos}. Hence, by using the exact same argument as in the proof Theorem \ref{cesaro}, one can show that, for $\alpha > (n-1)/2$,
$$
\lim_{N \to \infty}B^\alpha_N(S_\chi(c)) = S_\chi(c)
$$
in $E$ if $E$ is a TMIB and with respect to the weak-$\ast$ topology on $E$ if $E$ is a DTMIB.
\end{remark}
We denote by $c_{00}(\Lambda)$ the space consisting of all elements of $\C^\Lambda$ with only finitely many non-zero entries.
We  have the following consequence of Theorem \ref{cesaro}.
\begin{corollary} \label{cesaro-cor} Let $E$ be a TMIB. The space $c_{00}(\Lambda)$ is dense in  $E^B_d(\Lambda)$.
\end{corollary}

With the help of Theorem \ref{cesaro}, we can also describe the bilinear mapping \eqref{bil-synthesis-mappiing} from Proposition \ref{lemma-synthesis}.

\begin{corollary}\label{equ-for-bilinear-mapping-sum}
For all $c\in E^B_d(\Lambda)$ and $\varphi\in W(\mathcal{F}L^1_{\widetilde{\nu}_E},L^1_{\omega_E})$, it holds that
\begin{equation}\label{the-bilinear-mapping-syn-fom}
S_{\varphi}(c)=\sum_{\lambda\in\Lambda}c_{\lambda}T^B_{\lambda}\varphi
\end{equation}
where the series is C\'esaro summable in $E$ if $E$ is a TMIB and C\'esaro summable with respect to the weak-$\ast$ topology on $E$ if $E$ is a DTMIB.
\end{corollary}

\begin{proof} Since $\SSS(\R^n)$ is dense in $W(\mathcal{F}L^1_{\widetilde{\nu}_E},L^1_{\omega_E})$, Proposition \ref{lemma-synthesis} yields that \eqref{the-bilinear-mapping-syn-fom} holds true for all $c\in c_{00}(\Lambda)$ and $\varphi\in W(\mathcal{F}L^1_{\widetilde{\nu}_E},L^1_{\omega_E})$. Let $c\in E^B_d(\Lambda)$ be arbitrary. We define $c^{(N)}\in c_{00}(\Lambda)$, $N \in \Z_+$, by
$$
c^{(N)}_{A_{\Lambda}m}=
\left\{
\begin{array}{l}
\left(1-\frac{|m_1|}{N}\right)\cdots \left(1-\frac{|m_n|}{N}\right)c_{A_{\Lambda}m},\quad \mbox{if}\,\, |m_j|<N,\\
0,\quad \mbox{otherwise.}
\end{array}
\right.
$$
If $E$ is a TMIB, Proposition \ref{lemma-synthesis} and Theorem \ref{cesaro} imply that $S_{\varphi}(c^{(N)})$ converges to $S_{\varphi}(c)$ in $E$, which completes the proof in this case. Now suppose that $E$ is a DTMIB with $E=E'_0$, where $E_0$ is a TMIB. Let $\chi\in \DD(U)\backslash\{0\}$ be as in the second part of  Proposition \ref{lemma-synthesis}. Then, for every $g\in E_0$ there is $h\in E_0$ satisfying
\begin{align*}
&\left\langle \sum_{|m_j|<N} \left(1-\frac{|m_1|}{N}\right)\cdots \left(1-\frac{|m_n|}{N}\right)c_{A_{\Lambda}m}T^B_{A_{\Lambda}m}\varphi,g\right\rangle\\
&=\left\langle \sum_{|m_j|<N} \left(1-\frac{|m_1|}{N}\right)\cdots \left(1-\frac{|m_n|}{N}\right)c_{A_{\Lambda}m}T^B_{A_{\Lambda}m}\chi,h\right\rangle.
\end{align*}
Theorem \ref{cesaro} yields that the right-hand side of the above identity tends to $\langle S_{\chi}(c),h\rangle = \langle S_{\varphi}(c),g\rangle$. This completes the proof.
\end{proof}

\subsection{Duality and stability under completed tensor products of TMIB}\label{tensor} In this subsection, we determine the dual of the discrete space associated to a TMIB and study the discrete space associated to the completed tensor product of  two TMIB.
\begin{proposition}\label{dual-seq}
Let $E$ be a TMIB. The strong dual of $E^B_d(\Lambda)$ may be topologically identified with $(E')^{-B}_d(\Lambda)$ via the dual pairing
$$
\langle c', c \rangle = \sum_{\lambda \in \Lambda} c'_\lambda c_\lambda, \qquad c' \in (E')^{-B}_d(\Lambda),\, c \in E^B_d(\Lambda).
$$
Furthermore, the series $\sum_{\lambda \in \Lambda} c'_\lambda c_\lambda$ is C\'esaro summable in $\C$.
\end{proposition}
\begin{proof}
Let $\chi, \psi \in \mathcal{D}(U) \backslash \{0\}$ be such that $\int_{\R^d} \chi(x)\psi(x) dx = 1$. Note that
$$
 \sum_{\lambda \in \Lambda} c'_\lambda c_\lambda = \left\langle \sum_{\lambda \in \Lambda} c'_\lambda T^{-B}_\lambda \psi,  \sum_{\lambda \in \Lambda} c_\lambda T^{B}_\lambda \chi \right\rangle, \qquad  c' \in (E')^{-B}_d(\Lambda),\,  c \in c_{00}(\Lambda).
$$
Hence, Theorem \ref{cesaro} implies that the mapping
$$
(E')^{-B}_d(\Lambda) \rightarrow (E^B_d(\Lambda))'_b, \, c' \mapsto \left( c \mapsto  \sum_{\lambda \in \Lambda} c'_\lambda c_\lambda \right),
$$
is well-defined and continuous, and that the series $\sum_{\lambda \in \Lambda} c'_\lambda c_\lambda$ is C\'esaro summable in $\C$. This mapping is clearly injective. We now show that it is also surjective; the result then follows from the open mapping theorem. Pick $\chi  \in \mathcal{D}(U) \backslash \{0\}$ such that $\check{\chi}\in \DD(U)$ and set $\psi = \theta_{-B}(\check{\chi}) \in \mathcal{D}(U) \backslash \{0\}$. Let $x' \in  (E^B_d(\Lambda))'$ be arbitrary. There is $c' \in \C^\Lambda$ such that
$$
\langle x', c \rangle =  \sum_{\lambda \in \Lambda} c'_\lambda c_\lambda, \qquad c \in c_{00}(\Lambda).
$$
Since the space $c_{00}(\Lambda)$ is dense in $E^B_d(\Lambda)$ (Corollary \ref{cesaro-cor}), it suffices to show that $c' \in (E')^{-B}_d(\Lambda)$. Consider the continuous linear mapping $R_{\psi} : E \rightarrow E^{B}_d(\Lambda)$ from Proposition \ref{sampling} and denote its transpose by ${}^tR_{\psi}$. For all $\varphi \in \mathcal{D}(\R^n)$ it holds that
\begin{align*}
\langle {}^tR_{\psi}(x'), \varphi \rangle  &= \sum_{\lambda \in \Lambda} c'_\lambda \varphi \ast_B \psi(\lambda) \\
&=  \sum_{\lambda \in \Lambda} c'_\lambda \int_{\R^n} \varphi(x) T^{-B}_{\lambda}\chi(x) \dx \\
&= \int_{\R^n} \left ( \sum_{\lambda \in \Lambda} c'_\lambda T^{-B}_\lambda \chi \right)(x) \varphi(x) dx.
\end{align*}
Hence, $ \sum_{\lambda \in \Lambda} c'_\lambda T^{-B}_\lambda \chi =  {}^tR_{\psi}(x') \in E'$ and, thus, $c' \in (E')^{-B}_d(\Lambda)$.
\end{proof}

Our next goal is to show that the completed tensor product of the discrete spaces associated to two TMIB is canonically isomorphic to the discrete space associated to the completed tensor product of the two TMIB.

\begin{proposition}\label{thmforTenspr}
Let $E_j$ be a TMIB on $\R^{n_j}$, let $B_j$ be a real-valued $n_j\times n_j$-matrix, and let $\Lambda_j$ be a lattice in $\R^{n_j}$ for $j = 1,2$. Let $\tau$ denote either $\pi$ or $\epsilon$. Then, $(E_1)^{B_1}_d(\Lambda_1) \widehat{\otimes}_{\tau} (E_2)^{B_2}_d(\Lambda_2)$ is canonically isomorphic to $(E_1 \widehat{\otimes}_{\tau} E_2)^{B_1 \oplus B_2}_d(\Lambda_1 \times \Lambda_2)$.
\end{proposition}
\begin{proof}
Set $n=n_1+n_2$, $B=B_1\oplus B_2$, and $\Lambda = \Lambda_1 \times \Lambda_2$. Choose a bounded open neighbourhood $U_j$ of the origin in $\R^{n_j}$ such that the families of sets $\{ \lambda + U_j \, | \, \lambda \in \Lambda_j \}$, $j = 1,2$, are pairwise disjoint. Set $U = U_1 \times U_2$. Choose $\chi_j \in \mathcal{D}(U_j) \backslash \{0\}$ such that $\check{\chi}_j\in\DD(U_j)$, $j = 1,2$, and set $\chi = \chi_1 \otimes \chi_2 \in \mathcal{D}(U) \backslash \{0\}$. Denote by $\iota: (E_1)^{B_1}_d(\Lambda_1) \otimes_{\tau} (E_2)^{B_2}_d(\Lambda_2) \rightarrow (E_1 \widehat{\otimes}_{\tau} E_2)^B_d(\Lambda)$ the canonical inclusion mapping. We need to show that this mapping extends to a topological isomorphism from $(E_1)^{B_1}_d(\Lambda_1) \widehat{\otimes}_{\tau} (E_2)^{B_2}_d(\Lambda_2)$ onto $(E_1 \widehat{\otimes}_{\tau} E_2)^B_d(\Lambda)$. By the identity $S_{\chi_1} \otimes S_{\chi_2} = S_\chi \circ \iota$ and Corollary \ref{complemented}, it suffices to show that the mapping
$$
S_{\chi_1} \widehat{\otimes}_{\tau} S_{\chi_2}:  (E_1)^{B_1}_d(\Lambda_1) \widehat{\otimes}_{\tau} (E_2)^{B_2}_d(\Lambda_2) \rightarrow E_1 \widehat{\otimes}_{\tau} E_2
$$
is a topological embedding with range equal to $S_\chi((E_1 \hat{\otimes}_{\tau} E_2)^B_d(\Lambda))$. The mappings $S_{\chi_j}$, $j = 1,2$, are topological embeddings. Hence, by definition of the $\epsilon$-topology, the mapping $S_{\chi_1} \widehat{\otimes}_{\epsilon} S_{\chi_2}$ is a topological embedding as well (cf.\ \cite[p.\ 47]{ryan}). For the $\pi$-topology, Corollary \ref{complemented} and \cite[Proposition 2.4]{ryan} imply that also the mapping $S_{\chi_1} \widehat{\otimes}_{\pi} S_{\chi_2}$ is a topological embedding. The identity $S_{\chi_1} \otimes S_{\chi_2} = S_\chi \circ \iota$ and the fact that $S_\chi((E_1 \widehat{\otimes}_{\tau} E_2)^B_d(\Lambda))$ is closed in $E_1 \widehat{\otimes}_{\tau} E_2$ imply that the range of $S_{\chi_1} \widehat{\otimes}_{\tau} S_{\chi_2}$ is included in
$S_\chi((E_1 \widehat{\otimes}_{\tau} E_2)^B_d(\Lambda))$. As the space $c_{00}(\Lambda) = c_{00}(\Lambda_1) \otimes c_{00}(\Lambda_2) \subset  (E_1)^{B_1}_d(\Lambda_1) \otimes  (E_2)^{B_2}_d(\Lambda_2)$ is dense in $(E_1 \widehat{\otimes}_{\tau} E_2)^B_d(\Lambda)$ (Corollary \ref{cesaro-cor}) and $S_{\chi_1} \widehat{\otimes}_{\tau} S_{\chi_2}$ is a topological embedding, we conclude that the range of $S_{\chi_1} \widehat{\otimes}_{\tau} S_{\chi_2}$ is equal to $S_\chi((E_1 \widehat{\otimes}_{\tau} E_2)^B_d(\Lambda))$.
\end{proof}


\subsection{Examples} \label{sect-examples}
As explained in the introduction, the $2n \times 2n$-matrix $B_0$ from Definition \ref{def-B0} is the most important for our purposes. In this subsection, we  determine the discrete space associated to various TMIB and DTMIB with respect to  $B_0$. We start with the spaces considered in Example \ref{examples-TMIB}$(iii)$.

Let $\Lambda$ be a lattice in $\R^n$ and let $B$ be a real-valued $n\times n$ matrix. For $c\in \C^{\Lambda}$ we set $\theta_B(c)=(c_{\lambda}e^{2\pi i B\lambda\cdot\lambda})_{\lambda\in\Lambda}$. Given a Banach space $X\subset \C^{\Lambda}$, we define the Banach space $\theta_B X:=\{c\in \C^{\Lambda}\, |\, \theta_{-B}(c)\in X\}$ with norm $\|c\|_{\theta_B X}:=\|\theta_{-B}(c)\|_X$.
\begin{proposition}\label{mixed-norm-TMIB}
{Let $E$ be a TMIB on $\R^n$ and let $w$ be a polynomially bounded weight function on $\R^n$. Let $\Lambda_1$ and $\Lambda_2$ be two lattices in $\R^n$. }
\begin{itemize}
\item[$(i)$] It holds that
\begin{align*}
& (L^p_w(\R^n_\xi;E_x))^{B_0}_d (\Lambda_{1,x} \times \Lambda_{2,\xi}) = \ell^p_w(\Lambda_{2}; E^0_d(\Lambda_{1})), \qquad 1 \leq p < \infty, \\
&(C_{0,w}(\R^n_\xi;E_x))^{B_0}_d(\Lambda_{1,x} \times \Lambda_{2,\xi})= c_{0,w}(\Lambda_2;E^0_d(\Lambda_1)),
\end{align*}
topologically.
\item[$(ii)$] Suppose that $\nu_E =1$. Then,
\begin{align*}
&(L^p_w(\R^n_x;E_\xi))^{B_0}_d (\Lambda_{1,x} \times \Lambda_{2,\xi}) = \theta_{-B_0}\ell^p_w(\Lambda_1;E^0_d(\Lambda_2)), \qquad 1 \leq p < \infty, \\
&(C_{0,w}(\R^n_x;E_\xi))^{B_0}_d(\Lambda_{1,x} \times \Lambda_{2,\xi})= \theta_{-B_0}c_{0,w}(\Lambda_1;E^0_d(\Lambda_2)),
\end{align*}
topologically.
\end{itemize}
\end{proposition}
\begin{proof}
We only show the statements for $L^p_w$ as the proofs for $C_{0,w}$ are similar. Set $\Lambda=\Lambda_1\times \Lambda_2$. Choose a bounded open neighbourhood $U$ of the origin such that the families of sets $\{ \lambda_j + U \, | \, \lambda_j \in \Lambda_j \}$, $j= 1,2$,  are pairwise disjoint. Fix $\chi_j\in \DD(U)$ with $\chi_j(0) \neq 0$, $j=1,2$, and set $\chi=\chi_1\otimes \chi_2\in \DD(U\times U)\backslash\{0\}$.
\\ \\
$(i)$ Since $c_{00}(\Lambda)$  is dense in both $(L^p_w(\R^n_\xi;E_x))^{B_0}_d (\Lambda_1 \times \Lambda_2)$ and $\ell^p_w(\Lambda_2; E^0_d(\Lambda_1))$ (Corollary \ref{cesaro-cor}), it suffices to show that these spaces induce the same topology on $c_{00}(\Lambda)$.  For all $c=(c_{\lambda_1,\lambda_2})_{(\lambda_1,\lambda_2)\in\Lambda}\in c_{00}(\Lambda)$ it holds that
\begin{align*}
\| S_{\chi}(c) \|_{L^p_w(\R^n_\xi;E_x)}^p &=  \int_{\R^n} \left\|\sum_{\lambda_2\in\Lambda_2} T_{\lambda_2}\chi_2(\xi) \sum_{\lambda_1\in\Lambda_1}c_{\lambda_1,\lambda_2}e^{-2\pi i \lambda_1(\xi-\lambda_2)}T_{\lambda_1}\chi_1\right\|_E^p w(\xi)^p \dxi \\
&=  \int_{\R^n} \sum_{\lambda_2\in\Lambda_2} \left\| \sum_{\lambda_1\in\Lambda_1}c_{\lambda_1,\lambda_2}e^{-2\pi i \lambda_1(\xi-\lambda_2)}T_{\lambda_1}\chi_1 \right\|_E^p |T_{\lambda_2}\chi_2(\xi)|^p w(\xi)^p \dxi \\
&=  \int_{\R^n} \sum_{\lambda_2\in\Lambda_2} \left\| M_{-(\xi- \lambda_2)} \sum_{\lambda_1\in\Lambda_1}c_{\lambda_1,\lambda_2}T_{\lambda_1}  M_{\xi- \lambda_2}\chi_1 \right\|_E^p   |T_{\lambda_2}\chi_2(\xi)|^p w(\xi)^p \dxi.
\end{align*}
Since the set $\{ M_\eta \chi_1 \, | \, \eta \in U \}$ is bounded in $\mathcal{D}(U)$ and $\chi_1 \neq 0$, Lemma \ref{bound-bel} implies that there is $C> 0$ such that, for all $\eta \in U$ and $c \in c_{00}(\Lambda)$,
$$
C^{-1}\left\| \sum_{\lambda_1\in\Lambda_1}c_{\lambda_1,\lambda_2}T_{\lambda_1}  \chi_1\right\|_E \leq \left\| \sum_{\lambda_1\in\Lambda_1}c_{\lambda_1,\lambda_2}T_{\lambda_1}  M_{\eta}\chi_1\right\|_E \leq C \left\|\sum_{\lambda_1\in\Lambda_1} c_{\lambda_1,\lambda_2}T_{\lambda_1} \chi_1 \right\|_E.
$$
Hence,
\begin{align*}
\| S_{\chi}(c) \|_{L^p_w(\R^n_\xi;E_x)}^p &\leq \sum_{\lambda_2\in\Lambda_2} \int_{\R^n} \left\|  \sum_{\lambda_1\in\Lambda_1}c_{\lambda_1,\lambda_2}T_{\lambda_1}  M_{\xi- \lambda_2}\chi_1\right\|_E^p |T_{\lambda_2}(\chi_2 \nu_E)(\xi)|^p w(\xi)^p \dxi \\
&\leq C^p \sum_{\lambda_2\in\Lambda_2} \left\|\sum_{\lambda_1\in\Lambda_1}c_{\lambda_1,\lambda_2}T_{\lambda_1} \chi_1\right\|_E^p \int_{U}|\chi_2(\xi)|^p\nu_E(\xi)^p w(\xi+\lambda_2)^p \dxi\\
&\leq C'\sum_{\lambda_2\in\Lambda_2} \left\|\sum_{\lambda_1\in\Lambda_1}c_{\lambda_1,\lambda_2}T_{\lambda_1} \chi_1\right\|_E^p w(\lambda_2)^p = C' \| c \|^p_{\ell^p_w(\Lambda_{2}; E^0_d(\Lambda_{1}))}
\end{align*}
for all $c \in c_{00}(\Lambda)$. Next, choose an open neighbourhood $V$ of $0$ such that $ \inf_{\xi \in  V} |\chi_2(\xi)| > 0$. Then,
\begin{align*}
\| S_{\chi}(c) \|_{L^p_w(\R^n_\xi;E_x)}^p  & \geq \sum_{\lambda_2\in\Lambda_2} \int_{\R^n} \left\|\sum_{\lambda_1\in\Lambda_1} c_{\lambda_1,\lambda_2}T_{\lambda_1}  M_{\xi- \lambda_2}\chi_1\right\|_E^p |T_{\lambda_2}(\chi_2/ \check\nu_E)(\xi)|^p w(\xi)^p \dxi\\
&\geq C^{-p}\sum_{\lambda_2\in\Lambda_2} \left\|\sum_{\lambda_1\in\Lambda_1} c_{\lambda_1,\lambda_2}T_{\lambda_1} \chi_1\right\|_E^p  \int_{V}|\chi_2(\xi)|^p \check\nu_E(\xi)^{-p} w(\xi+\lambda_2)^p \dxi\\
&\geq C'^{-1} \sum_{\lambda_2\in\Lambda_2}\left\|\sum_{\lambda_1\in\Lambda_1} c_{\lambda_1,\lambda_2}T_{\lambda_1} \chi_1\right\|_E^p w(\lambda_2)^p = C'^{-1} \| c \|^p_{\ell^p_w(\Lambda_{2}; E^0_d(\Lambda_{1}))}
\end{align*}
for all  $c \in c_{00}(\Lambda)$. This shows the result.
\\ \\
$(ii)$ As in part $(i)$, it suffices to show that the spaces  $(L^p_w(\R^n_x;E_\xi))^{B_0}_d (\Lambda_{1,x} \times \Lambda_{2,\xi})$ and $\theta_{-B_0}\ell^p_w(\Lambda_1;E^0_d(\Lambda_2))$ induce the same topology on $c_{00}(\Lambda)$. Let $c\in c_{00}(\Lambda)$ be arbitrary and set $\tilde{c}=\theta_{B_0}(c)$. Then,
$$
S_{\chi}(c)=\sum_{\lambda_1\in\Lambda_1}T_{\lambda_1}\chi_1\otimes \left(M_{-\lambda_1}\sum_{\lambda_2\in\Lambda_2}\tilde{c}_{\lambda_1,\lambda_2} T_{\lambda_2}\chi_2\right)
$$
and thus
$$
\|S_{\chi}(c)\|^p_{L^p_w(\R^n_x;E_\xi)}=\sum_{\lambda_1\in\Lambda_1} \|T_{\lambda_1}\chi_1\|^p_{L^p_w} \left\|M_{-\lambda_1}\sum_{\lambda_2\in\Lambda_2}\tilde{c}_{\lambda_1,\lambda_2} T_{\lambda_2}\chi_2\right\|^p_E.
$$
We infer that
$$
C^{-1} \|\tilde{c}\|^p_{\ell^p_w(\Lambda_1;E^0_d(\Lambda_2))} \leq \|S_{\chi}(c)\|_{L^p_w(\R^n_x;E_\xi)} \leq C \|\tilde{c}\|^p_{\ell^p_w(\Lambda_1;E^0_d(\Lambda_2))},
$$
from which the result follows.
\end{proof}

\begin{remark}\label{mixed-norm-TMIB-remark}
If $E$, $w$, $\Lambda_1$ and $\Lambda_2$ are as in Proposition \ref{mixed-norm-TMIB}$(i)$, the exact same argument as in its proof shows that
$$
(L^p_w(\R^n_\xi;E_x))^{-B_0}_d (\Lambda_{1,x} \times \Lambda_{2,\xi}) = \ell^p_w(\Lambda_2; E^0_d(\Lambda_1)), \qquad 1\leq p<\infty,
$$
topologically.
\end{remark}

\begin{proposition}
Let $E$ be a DTMIB with the Radon-Nikod\'ym property and let $w$ be a polynomially bounded weight function on $\R^n$. Let $\Lambda_1$ and $\Lambda_2$ be two lattices in $\R^n$.
\begin{itemize}
\item[$(i)$] It holds that
$$
(L^p_w(\R^n_\xi;E_x))^{B_0}_d (\Lambda_{1,x} \times \Lambda_{2,\xi}) = \ell^p_w(\Lambda_{2}; E^0_d(\Lambda_{1})), \qquad 1 < p \leq \infty,
$$
 topologically.
 \item[$(ii)$] Suppose that $\nu_E= 1$.  Then,
\begin{gather*}
(L^p_w(\R^n_x;E_\xi))^{B_0}_d (\Lambda_{1,x} \times \Lambda_{2,\xi}) = \theta_{-B_0}\ell^p_w(\Lambda_1;E^0_d(\Lambda_2)), \qquad 1< p \leq \infty,
\end{gather*}
topologically.
\end{itemize}
 \end{proposition}
\begin{proof}
We only show $(i)$ as $(ii)$ can be treated similarly. Suppose that $E = E'_0$, where $E_0$ is a TMIB. Let $q$ be the H\"older conjugate index to $p$. As $E$ satisfies the Radon-Nikod\'ym property, we have that
 $$
L^p_w(\R^n;E) = (L^q_{1/w}(\R^n;E_0))'.
 $$
 Proposition \ref{dual-seq} yields that $E^0_d(\Lambda_{1}) = ((E_0)^0_d(\Lambda_1))'$.
Since, by Corollary \ref{complemented}, $E^0_d(\Lambda_{1})$ also satisfies the Radon-Nikod\'ym property, we  have that
$$
\ell^p_w(\Lambda_{2}; E^0_d(\Lambda_{1})) =  (\ell^q_{1/w}(\Lambda_{2}; (E_0)^0_d(\Lambda_{1})))'.
$$
The result now follows from Proposition \ref{dual-seq} and \ref{mixed-norm-TMIB-remark}.
\end{proof}

Next, we discuss completed tensor products of two TMIB on $\R^n$.  Note that $B_0$ is not the direct sum of two $n \times n$-matrices and therefore this matrix is not covered by Proposition \ref{thmforTenspr}. It would be interesting to determine the discrete space associated to the completed tensor product of two TMIB with respect to $B_0$ but even for most of the $L^p$-spaces we do not know how to do this:

\begin{problem}\label{tensor-problem}
Let $\Lambda_1$ and $\Lambda_2$ be lattices in $\R^n$. Let $\tau$ denote either $\pi$ or $\epsilon$. Give an explicit description of $(L^{p_1} \widehat{\otimes}_{\tau} L^{p_2})^{B_0}_d(\Lambda_{1} \times \Lambda_{2})$ for $1 \leq p_1, p_2 < \infty$.
\end{problem}

For $\tau = \pi$ and $p_1 = 1$ we have the following (trivial) answer to Problem \ref{tensor-problem} (cf.\ \cite[Section 2.3]{ryan}):
$$
(L^1 \widehat{\otimes}_\pi L^{p_2})^{B_0}_d(\Lambda_{1} \times \Lambda_{2}) = (L^1(L^{p_2}))^{B_0}_d(\Lambda_{1} \times \Lambda_{2}) = \ell^1(\Lambda_1;\ell^{p_2}(\Lambda_2)),
$$
where the second equality follows from Example \ref{examples-sequences}$(i)$.

We now provide an answer to Problem \ref{tensor-problem} for $p_2=2$ and $p_1$ varying in a certain range. In fact, we are able to show the following more general result.

\begin{proposition}\label{seq-spa-for-the-pi-prod}
Let $\Lambda_1$ and $\Lambda_2$ be two lattices in $\R^n$, let $w_1$ and $w_2$ be two polynomially bounded weight functions on $\R^n$, and let $1 \leq p_1, p_2 < \infty$. Then,
\begin{itemize}
\item[$(i)$] $(L^{p_1}_{w_1}\widehat{\otimes}_{\pi} \mathcal{F}L^{p_2}_{w_2})^{B_0}_d(\Lambda_1\times \Lambda_2)=\theta_{-B_0}\ell^1_{w_1w_2}(\Lambda_1;(\mathcal{F}L^{p_2})^0_d(\Lambda_2))$  if $p^{-1}_1+p^{-1}_2 \geq  1$,
\item[$(ii)$] $(L^{p_1}_{w_1}\widehat{\otimes}_{\epsilon} \mathcal{F}L^{p_2}_{w_2})^{B_0}_d(\Lambda_1\times \Lambda_2)=\theta_{-B_0}c_{0,w_1w_2}(\Lambda_1;(\mathcal{F}L^{p_2})^0_d(\Lambda_2))$ if $p^{-1}_1+p^{-1}_2 \leq  1$,
\end{itemize}
topologically.
\end{proposition}
\begin{proof} Set $\Lambda=\Lambda_1\times\Lambda_2$. Let $U\subset \R^n$ be a bounded open neighbourhood of the origin such that the families $\{\lambda_j+U\, |\, \lambda_j\in \Lambda_j\}$, $j=1,2$, are pairwise disjoint. Choose $A_j, \kappa_j > 0$ such that $w_j(x+y) \leq A_jw_j(x)(1+|y|)^{\kappa_j}$, $j=1,2$, for all $x,y\in\R^n$. We write $q_j$ for the H\"older conjugate index to $p_j$, $j = 1,2$. By the closed graph theorem, it suffices to show that the identities in $(i)$ and $(ii)$ hold algebraically.\\
\\
\noindent $(i)$ We first show that $\theta_{-B_0}\ell^1_{w_1w_2}(\Lambda_1;(\mathcal{F}L^{p_2})^0_d(\Lambda_2)) \subseteq (L^{p_1}_{w_1}\widehat{\otimes}_{\pi} \mathcal{F}L^{p_2}_{w_2})^{B_0}_d(\Lambda)$. Pick $\chi_1,\chi_2\in\DD(U)\backslash\{0\}$ and set $\chi=\chi_1\otimes\chi_2\in\DD(U\times U)\backslash\{0\}$. Let $c \in \theta_{-B_0}\ell^1_{w_1w_2}(\Lambda_1;(\mathcal{F}L^{p_2})^0_d(\Lambda_2))$ be arbitrary and set $\tilde{c}=\theta_{B_0}(c)\in \ell^1_{w_1w_2}(\Lambda_1;(\mathcal{F}L^{p_2})^0_d(\Lambda_2))$. We have that
\begin{equation}\label{equ-ex-for-map-ser}
S_{\chi}(c)=\sum_{\lambda_1\in\Lambda_1}T_{\lambda_1}\chi_1 \otimes \left(M_{-\lambda_1}\sum_{\lambda_2\in\Lambda_2} \tilde{c}_{\lambda_1,\lambda_2}T_{\lambda_2}\chi_2\right).
\end{equation}
Corollary \ref{cor-for-equ-wight-notwei-four-spa} implies that, for $\lambda_1\in\Lambda_1$ fixed, $(\tilde{c}_{\lambda_1,\lambda_2})_{\lambda_2\in\Lambda_2}\in (\mathcal{F}L^{p_2}_{w_2})^0_d(\Lambda_2)$ and that
\begin{align*}
\left\|M_{-\lambda_1}\sum_{\lambda_2\in\Lambda_2} \tilde{c}_{\lambda_1,\lambda_2} T_{\lambda_2}\chi_2\right\|_{\mathcal{F}L^{p_2}_{w_2}} &\leq A_2w_2(\lambda_1) \|(\tilde{c}_{\lambda_1,\lambda_2})_{\lambda_2\in\Lambda_2} \|_{(\mathcal{F}L^{p_2}_{(1+|\,\cdot \,|)^{\kappa_2}})^0_d(\Lambda_2)}\\
&\leq Cw_2(\lambda_1) \|(\tilde{c}_{\lambda_1,\lambda_2})_{\lambda_2\in\Lambda_2} \|_{(\mathcal{F}L^{p_2})^0_d(\Lambda_2)}.
\end{align*}
We obtain that the series in the right-hand side of \eqref{equ-ex-for-map-ser} (over $\Lambda_1$) is absolutely summable in $L^{p_1}_{w_1}\widehat{\otimes}_{\pi} \mathcal{F}L^{p_2}_{w_2}$. This shows the desired inclusion. Next, we prove that $(L^{p_1}_{w_1}\widehat{\otimes}_{\pi} \mathcal{F}L^{p_2}_{w_2})^{B_0}_d(\Lambda)\subseteq \theta_{-B_0}\ell^1_{w_1w_2}(\Lambda_1;(\mathcal{F}L^{p_2})^0_d(\Lambda_2))$. Let $r>0$ be such that $[-2r,2r]^n\subset U$. Pick $\psi\in\DD_{[-r,r]^n}\backslash\{0\}$ and set
$$
\psi_1(x,\xi)=e^{-2\pi i x\cdot\xi}\psi(x)\psi(\xi)\quad \mbox{and}\quad \psi_2=(\psi_1\ast_{B_0}\psi_1)\check.
$$
We choose $\psi$ so  that $\psi_2$ is not the zero function. Then, $\psi_1\in\DD_{[-r,r]^{2n}}\backslash\{0\}$ and $\psi_2\in\DD(U\times U)\backslash\{0\}$. Furthermore, choose $\chi\in \DD(U\times U)$ such that $(\theta_{B_0}(\psi_2),\overline{\chi})_{L^2}=1$. Corollary \ref{complemented} implies that
\begin{equation}\label{equ-for-equ-of-seq-s}
c=R_{\check{\psi}_2}(S_{\chi}(c))\in R_{\check{\psi}_2}(L^{p_1}_{w_1}\widehat{\otimes}_{\pi} \mathcal{F}L^{p_2}_{w_2}),\qquad c\in(L^{p_1}_{w_1}\widehat{\otimes}_{\pi} \mathcal{F}L^{p_2}_{w_2})^{B_0}_d(\Lambda).
\end{equation}
We claim that
\begin{equation}\label{equ-for-the-claim-of-s}
F\ast_{B_0}\psi_1\in L^1_{w_1w_2}(\mathcal{F}L^{p_2}),\qquad F\in L^{p_1}_{w_1} \widehat{\otimes}_{\pi} \mathcal{F}L^{p_2}_{w_2}.
\end{equation}
Before we prove \eqref{equ-for-the-claim-of-s}, let us show how it entails the result. For all $ F\in L^{p_1}_{w_1} \widehat{\otimes}_{\pi} \mathcal{F}L^{p_2}_{w_2}$, we have that
$$
R_{\check{\psi}_2}(F)=(F\ast_{B_0}(\psi_1\ast_{B_0}\psi_1)(\lambda))_{\lambda\in\Lambda} =((F\ast_{B_0}\psi_1)\ast_{B_0}\psi_1(\lambda))_{\lambda\in\Lambda}=R_{\psi_1}(F\ast_{B_0} \psi_1).
$$
Hence, in view of  \eqref{equ-for-equ-of-seq-s} and \eqref{equ-for-the-claim-of-s}, the desired inclusion follows from Proposition \ref{sampling} and Proposition \ref{mixed-norm-TMIB}$(ii)$. It remains to prove \eqref{equ-for-the-claim-of-s}. Let $F\in L^{p_1}_{w_1}\widehat{\otimes}_{\pi} \mathcal{F}L^{p_2}_{w_2}$ be arbitrary. Then,
\begin{equation}\label{ser-rep-for-capf}
F=\sum_{j=0}^{\infty}a_j f_j\otimes g_j,
\end{equation}
where $f_j,g_j\in\SSS(\R^n)$ are such that $(f_j)_{j\in\N}$ is bounded in $L^{p_1}_{w_1}$ and $(g_j)_{j\in\N}$ is bounded in $\mathcal{F}L^{p_2}_{w_2}$, and $a_j\in \C$ are such that $\sum_{j=0}^{\infty}|a_j|<\infty$ (cf.\ \cite[Proposition 2.8]{ryan}). For all $f,g\in\SSS(\R^n)$ it holds that
\begin{align}
(f\otimes g)\ast_{B_0}\psi_1(t,\eta)&=\iint_{\R^{2n}} f(x)g(\xi) \psi(t-x)\psi(\eta-\xi) e^{-2\pi i t\cdot(\eta-\xi)}\dx\dxi \label{equ-for-con-with-gff}\\
&=f*\psi(t) g*(M_{-t}\psi)(\eta).\nonumber
\end{align}
We estimate as follows
\begin{align*}
\|&(f\otimes g)\ast_{B_0}\psi_1\|_{L^1_{w_1w_2}(\mathcal{F}L^{p_2})}\\
&= \int_{\R^n}|f*\psi(t)|w_1(t)w_2(t)\|(\mathcal{F}^{-1}g)(T_t\mathcal{F}^{-1} \psi)\|_{L^{p_2}}\dt\\
&\leq A_2\|f*\psi\|_{L^{q_2}_{w_1}} \left(\iint_{\R^{2n}}|\mathcal{F}^{-1}g(\xi)|^{p_2}w_2(\xi)^{p_2} |\mathcal{F}^{-1}\psi(\xi-t)|^{p_2}(1+|t-\xi|)^{\kappa_2 p_2}\dt\dxi\right)^{1/p_2}\\
&= A_2\|f*\psi\|_{L^{q_2}_{w_1}} \|g\|_{\mathcal{F}L^{p_2}_{w_2}}\|\psi\|_{\mathcal{F}L^{p_2}_{(1+|\, \cdot \,|)^{\kappa_2}}} \\
&\leq C \|f\|_{L^{p_1}_{w_1}} \|g\|_{\mathcal{F}L^{p_2}_{w_2}},
\end{align*}
where the last inequality follows from  Young's inequality (note that $q_2\geq p_1$). The representation \eqref{ser-rep-for-capf} and the above estimate yield that $F\ast_{B_0}\psi_1\in L^1_{w_1w_2}(\mathcal{F}L^{p_2})$.
\\ \\
\noindent $(ii)$ The assumption on $p_1$ and $p_2$ implies that $1 < p_1, p_2 < \infty$ and thus also $1 < q_1, q_2 < \infty$.  First we prove that $\theta_{-B_0}c_{0,w_1w_2}(\Lambda_1;(\mathcal{F}L^{p_2})^0_d(\Lambda_2)) \subseteq (L^{p_1}_{w_1}\widehat{\otimes}_{\epsilon} \mathcal{F}L^{p_2}_{w_2})^{B_0}_d(\Lambda)$. Choose $\chi_1,\tilde{\chi}_2\in\DD(U)\backslash\{0\}$ such that $\chi_2= \tilde{\chi}_2\ast \tilde{\chi}_2\in \DD(U)\backslash\{0\}$ and set $\chi=\chi_1\otimes\chi_2\in\DD(U\times U)\backslash\{0\}$. Let $c\in \theta_{-B_0}c_{0,w_1w_2}(\Lambda_1;(\mathcal{F}L^{p_2})^0_d(\Lambda_2))$ be arbitrary and set $\tilde{c}=\theta_{B_0}(c)$. Then, the representation \eqref{equ-ex-for-map-ser} holds true and, as in part $(i)$, Corollary \ref{cor-for-equ-wight-notwei-four-spa} implies that
$$
g_{\lambda_1} = M_{-\lambda_1}\sum_{\lambda_2\in\Lambda_2} \tilde{c}_{\lambda_1,\lambda_2}T_{\lambda_2}\chi_2 \in \mathcal{F}L^{p_2}_{w_2},\quad \lambda_1\in\Lambda_1.
$$
We now show that the series in the right-hand side of \eqref{equ-ex-for-map-ser} (over $\Lambda_1$) is unconditionally convergent in $L^{p_1}_{w_1}\widehat{\otimes}_{\epsilon} \mathcal{F}L^{p_2}_{w_2}$. Denote by $K_1$ and $K_2$ the closed unit balls in $L^{q_1}_{1/w_1}=(L^{p_1}_{w_1})'$ and $\mathcal{F}L^{q_2}_{1/\check{w}_2}=(\mathcal{F}L^{p_2}_{w_2})'$, respectively. Set
$$
A_3 =\sup_{x\in U}(1+|x|)^{\kappa_1},\quad A_4=\left(\sup_{x\in\R^n} \sum_{\lambda_1\in\Lambda_1} |\widehat{\tilde{\chi}}_2(x+\lambda_1)|^{q_2} (1+|x+\lambda_1|)^{\kappa_2q_2}\right)^{1/q_2}.
$$
Let $\varepsilon>0$ be arbitrary. As $\tilde{c}\in c_{0,w_1w_2}(\Lambda_1;(\mathcal{F}L^{p_2})^0_d(\Lambda_2))$, there is a finite subset $\Lambda^{(0)}_1$ of $\Lambda_1$ such that, for all $\lambda_1\in\Lambda_1\backslash\Lambda^{(0)}_1$,
$$
w_1(\lambda_1)w_2(\lambda_1)\left\|\sum_{\lambda_2\in\Lambda_2} \tilde{c}_{\lambda_1,\lambda_2} T_{\lambda_2}\tilde{\chi}_2\right\|_{\mathcal{F}L^{p_2}}\leq (2\|\chi_1\|_{L^{p_1}}A_1A_2 A_3A_4)^{-1}\varepsilon =:\varepsilon_1.
$$
For any $\Lambda_1^{(0)}\subseteq \Lambda'_1,\Lambda''_1\subset \Lambda_1$, $\Lambda'_1$ and $\Lambda''_1$ finite, and $f_1\in K_1$, $f_2\in K_2$, we have that
\begin{align*}
&\left|\left\langle f_1\otimes f_2, \sum_{\lambda_1\in \Lambda'_1} T_{\lambda_1}\chi_1\otimes g_{\lambda_1}- \sum_{\lambda_1\in \Lambda''_1} T_{\lambda_1}\chi_1\otimes g_{\lambda_1} \right\rangle\right|\\
&\leq \sum_{\lambda_1\in \Lambda'_1\backslash \Lambda^{(0)}_1} |\langle f_1,T_{\lambda_1}\chi_1\rangle| |\langle f_2, g_{\lambda_1}\rangle| + \sum_{\lambda_1\in \Lambda''_1\backslash \Lambda^{(0)}_1} |\langle f_1,T_{\lambda_1}\chi_1\rangle| |\langle f_2, g_{\lambda_1}\rangle|.
\end{align*}
Denote these sums by $I'$ and $I''$, respectively. We estimate $I'$ as follows
\begin{align*}
I'&\leq \|\chi_1\|_{L^{p_1}}\sum_{\lambda_1\in \Lambda'_1\backslash \Lambda^{(0)}_1} \|f_1\|_{L^{q_1}(\lambda_1+U)} \|(M_{-\lambda_1}f_2)\ast\check{\tilde{\chi}}_2\|_{\mathcal{F}L^{q_2}} \left\|\sum_{\lambda_2\in\Lambda_2} \tilde{c}_{\lambda_1,\lambda_2} T_{\lambda_2}\tilde{\chi}_2\right\|_{\mathcal{F}L^{p_2}}\\
&\leq \varepsilon_1\|\chi_1\|_{L^{p_1}}\sum_{\lambda_1\in \Lambda'_1\backslash \Lambda^{(0)}_1} \frac{\|f_1\|_{L^{q_1}(\lambda_1+U)}}{w_1(\lambda_1)}\cdot \frac{\|\mathcal{F}^{-1}f_2 T_{-\lambda_1}\widehat{\tilde{\chi}}_2\|_{L^{q_2}}}{w_2(\lambda_1)}\\
&\leq \varepsilon_1\|\chi_1\|_{L^{p_1}}\left(\sum_{\lambda_1\in \Lambda_1} \frac{\|f_1\|^{p_2}_{L^{q_1}(\lambda_1+U)}}{w_1(\lambda_1)^{p_2}}\right)^{1/p_2} \left(\sum_{\lambda_1\in \Lambda_1} \frac{\|\mathcal{F}^{-1}f_2 T_{-\lambda_1}\widehat{\tilde{\chi}}_2\|^{q_2}_{L^{q_2}}} {w_2(\lambda_1)^{q_2}}\right)^{1/q_2}.
\end{align*}
Since $q_1\leq p_2$, we infer that
\begin{align*}
\left(\sum_{\lambda_1\in \Lambda_1} \frac{\|f_1\|^{p_2}_{L^{q_1}(\lambda_1+U)}}{w_1(\lambda_1)^{p_2}}\right)^{1/p_2} &\leq A_1A_3\left(\sum_{\lambda_1\in \Lambda_1}\|f_1\|^{p_2}_{L^{q_1}_{1/w_1}(\lambda_1+U)}\right)^{1/p_2} \\
&\leq A_1 A_3\|f_1\|_{L^{q_1}_{1/w_1}} \leq A_1A_3.
\end{align*}
Furthermore,
\begin{align*}
&\sum_{\lambda_1\in \Lambda_1} \frac{\|\mathcal{F}^{-1}f_2 T_{-\lambda_1}\widehat{\tilde{\chi}}_2\|^{q_2}_{L^{q_2}}} {w_2(\lambda_1)^{q_2}}\\
&\leq A^{q_2}_2 \int_{\R^n} \frac{|\mathcal{F}^{-1}f_2(\xi)|^{q_2}}{\check{w}_2(\xi)^{q_2}} \sum_{\lambda_1\in \Lambda_1}|\widehat{\tilde{\chi}}_2(\xi+\lambda_1)|^{q_2} (1+|\xi+\lambda_1|)^{\kappa_2q_2}\dxi  \\
&\leq A^{q_2}_2A^{q_2}_4\|f_2\|^{q_2}_{\mathcal{F}L^{q_2}_{1/\check{w}_2}} \leq A^{q_2}_2A^{q_2}_4.
\end{align*}
Plugging these bounds into the above estimate for $I'$, we deduce that $I'\leq \varepsilon/2$. Analogously, we find that $I''\leq \varepsilon/2$. Hence,
$$
\sup_{f_1\in K_1}\sup_{f_2\in K_2}\left|\left\langle f_1\otimes f_2, \sum_{\lambda_1\in \Lambda'_1} T_{\lambda_1}\chi_1\otimes g_{\lambda_1}- \sum_{\lambda_1\in \Lambda''_1} T_{\lambda_1}\chi_1\otimes g_{\lambda_1} \right\rangle\right|\leq \varepsilon,
$$
from which the statement and therefore also the desired inclusion follows. Finally, we prove that
$(L^{p_1}_{w_1}\widehat{\otimes}_{\epsilon} \mathcal{F}L^{p_2}_{w_2})^{B_0}_d(\Lambda)\subseteq  \theta_{-B_0}c_{0,w_1w_2}(\Lambda_1;(\mathcal{F}L^{p_2})^0_d(\Lambda_2))$. Let $\psi$, $\psi_1$ and $\psi_2$ be as in the second part of the proof of part $(i)$. Pick $\tilde{\chi}_1,\chi_2\in\DD(U)\backslash\{0\}$ such that $\chi_1=\tilde{\chi}_1\ast\tilde{\chi}_1\ast\tilde{\chi}_1\in\DD(U)\backslash\{0\}$ and set $\chi=\chi_1\otimes \chi_2\in\DD(U\times U)\backslash\{0\}$. Choose $\tilde{\chi}_1$ and $\chi_2$ such that $(\theta_{B_0}(\psi_2),\overline{\chi})_{L^2}=1$. Corollary \ref{complemented} implies that
\begin{equation}\label{equ-for-equ-of-seq-s1}
c=R_{\check{\psi}_2}(S_{\chi}(c)),\qquad c\in(L^{p_1}_{w_1}\widehat{\otimes}_{\epsilon} \mathcal{F}L^{p_2}_{w_2})^{B_0}_d(\Lambda).
\end{equation}
Arguing as in the proof of part $(i)$, we see that it suffices to show that
\begin{equation}\label{equ-for-prov-eqpps}
S_{\chi}(c)\ast_{B_0}\psi_1\in C_{0,w_1w_2}(\mathcal{F}L^{p_2}),\qquad c\in (L^{p_1}_{w_1}\widehat{\otimes}_{\epsilon} \mathcal{F}L^{p_2}_{w_2})^{B_0}_d(\Lambda).
\end{equation}
Without loss of generality, we may assume that $w_1$ and $w_2$ are continuous. Then, (cf. \cite[Section 3.1]{ryan})
\begin{equation}\label{eps-prod-repr}
C_{0,w_1}(\mathcal{F}L^{p_2}_{w_2}) = C_{0,w_1} \widehat{\otimes}_\epsilon \mathcal{F}L^{p_2}_{w_2}.
\end{equation}
Given a continuous polynomially bounded weight function $w$ on $\R^n$, we denote by $J_w$ the isometrical isomorphism $J_w: C_0\rightarrow C_{0,w}, \, \varphi \mapsto \varphi/w$. Then, ${}^tJ_w: (C_{0,w})'\rightarrow \mathcal{M}^1$ is an isometrical isomorphism. Set $\tilde{\chi}=\tilde{\chi}_1\otimes\chi_2\in\DD(U\times U)\backslash\{0\}$. Let $c\in c_{00}(\Lambda)$ be arbitrary. For all $f_1\in (C_{0,w_1})'$ and $f_2\in (\mathcal{F}L^{p_2}_{w_2})'$, it holds that
\begin{equation}\label{est-for-change-of-spa-to-inf}
\langle f_1\otimes f_2, S_{\chi}(c)\rangle=\langle (f_1\ast\check{\tilde{\chi}}_1\ast\check{\tilde{\chi}}_1)\otimes f_2, S_{\tilde{\chi}}(c)\rangle.
\end{equation}
By Young's inequality, we infer that
\begin{align*}
&\|(f_1\ast\check{\tilde{\chi}}_1)\ast\check{\tilde{\chi}}_1\|_{L^{q_1}_{1/w_1}} \leq A_1 \|\check{\tilde{\chi}}_1\|_{L^{q_1}_{(1+ |\, \cdot \, |)^{\kappa_1}}} \|f_1\ast\check{\tilde{\chi}}_1\|_{L^1_{1/w_1}} \\
&\leq  A^2_1 \|\check{\tilde{\chi}}_1\|_{L^{1}_{(1+ |\, \cdot \, |)^{\kappa_1}}} \|\check{\tilde{\chi}}_1\|_{L^{q_1}_{(1+ |\, \cdot \, |)^{\kappa_1}}} \| {}^tJ_{w_1} f_1\|_{\mathcal{M}^1} =  A^2_1 \|\check{\tilde{\chi}}_1\|_{L^{1}_{(1+ |\, \cdot \, |)^{\kappa_1}}} \|\check{\tilde{\chi}}_1\|_{L^{q_1}_{(1+ |\, \cdot \, |)^{\kappa_1}}}  \|f_1\|_{(C_{0,w_1})'}.
\end{align*}
Hence, in view of \eqref{eps-prod-repr}, \eqref{est-for-change-of-spa-to-inf} and the fact that $c_{00}(\Lambda)$ is dense in $(L^{p_1}_{w_1}\widehat{\otimes}_{\epsilon} \mathcal{F}L^{p_2}_{w_2})^{B_0}_d(\Lambda)$ (Corollary \ref{cesaro-cor}), we deduce that the mapping
$$
(L^{p_1}_{w_1}\widehat{\otimes}_{\epsilon} \mathcal{F}L^{p_2}_{w_2})^{B_0}_d(\Lambda)\rightarrow C_{0,w_1}( \mathcal{F}L^{p_2}_{w_2}),\, c\mapsto S_{\chi}(c),
$$
is well-defined and continuous. Consequently, to prove \eqref{equ-for-prov-eqpps}, it suffices to show that  the mapping
 $$
C_{0,w_1}(\mathcal{F}L^{p_2}_{w_2}) \rightarrow C_{0,w_1w_2}(\mathcal{F}L^{p_2}),\, F\mapsto F\ast_{B_0}\psi_1,
$$
is well-defined and continuous. Let $F \in \SSS(\R^{2n})$ be arbitrary. Note that (cf.\ \eqref{equ-for-con-with-gff})
\begin{align*}
\operatorname{id}\widehat{\otimes}\mathcal{F}^{-1}(F\ast_{B_0}\psi_1)(t,\eta)&= \iint_{\R^{2n}} F(x,\xi)\psi(t-x)\widehat{\psi}(t-\eta)e^{2\pi i \xi\cdot \eta} \dx\dxi\\
&=\widehat{\psi}(t-\eta)\int_{\R^n} \operatorname{id}\widehat{\otimes}\mathcal{F}^{-1}(F)(x,\eta) \psi(t-x)\dx.
\end{align*}
We infer that
\begin{align*}
& \| F \ast_{B_0} \psi_1 \|_{L^\infty_{w_1w_2}(\mathcal{F}L^{p_2})} \\
&\leq \sup_{t \in \R^n} w_1(t)w_2(t)\left ( \int_{\R^n} |\widehat{\psi}(t-\eta)|^{p_2} \left(\int_{\R^n}|\operatorname{id}\widehat{\otimes}\mathcal{F}^{-1}(F)(x,\eta)| |\psi(t-x)|\dx \right)^{p_2}  {\rm d}\eta\right)^{1/p_2}\\
&\leq A_1A_2\|\psi\|_{L^{\infty}_{(1+|\, \cdot \,|)^{\kappa_1}}}\|\widehat{\psi}\|_{L^{\infty}_{(1+|\, \cdot \,|)^{\kappa_2}}} \times \\
&\phantom{\leq} \sup_{t \in \R^n} \left ( \int_{\R^n} \left(\int_{t-U}|\operatorname{id}\widehat{\otimes}\mathcal{F}^{-1}(F)(x,\eta)|w_1(x) \dx \right)^{p_2}  w_2(\eta)^{p_2} {\rm d}\eta\right)^{1/p_2}\\
&\leq A_1A_2 |U|^{1/q_2}\|\psi\|_{L^{\infty}_{(1+|\, \cdot \,|)^{\kappa_1}}}\|\widehat{\psi}\|_{L^{\infty}_{(1+|\, \cdot \,|)^{\kappa_2}}} \times \\
&\phantom{\leq} \sup_{t \in \R^n} \left ( \int_{\R^n} \int_{t-U}|\operatorname{id}\widehat{\otimes}\mathcal{F}^{-1}(F)(x,\eta)|^{p_2}w_1(x)^{p_2}   w_2(\eta)^{p_2} \dx {\rm d}\eta\right)^{1/p_2}\\
&\leq A_1A_2 |U|\|\psi\|_{L^{\infty}_{(1+|\, \cdot \,|)^{\kappa_1}}}\|\widehat{\psi}\|_{L^{\infty}_{(1+|\, \cdot \,|)^{\kappa_2}}} \| \operatorname{id}\widehat{\otimes}\mathcal{F}^{-1}(F)\|_{L^\infty_{w_1}(L^{p_2}_{w_2})}\\
&= A_1A_2 |U|\|\psi\|_{L^{\infty}_{(1+|\, \cdot \,|)^{\kappa_1}}}\|\widehat{\psi}\|_{L^{\infty}_{(1+|\, \cdot \,|)^{\kappa_2}}} \| F\|_{L^\infty_{w_1}(\mathcal{F}L^{p_2}_{w_2})}.
\end{align*}
The statement now follows from the density of $\SSS(\R^{2n})$  in $C_{0,w_1}(\mathcal{F}L^{p_2}_{w_2})$.
\end{proof}

\begin{corollary}
Let $\Lambda_1$ and $\Lambda_2$ be two lattices in $\R^n$ and let $w$ be a polynomially bounded weight function on $\R^n$. Then,
\begin{itemize}
\item[$(i)$] $(L^p_w\widehat{\otimes}_{\pi} L^2)^{B_0}_d(\Lambda_1\times\Lambda_2)=\ell^1_w(\Lambda_1;\ell^2(\Lambda_2))$ if $1 \leq p \leq 2$,
\item[$(ii)$] $(L^p_w\widehat{\otimes}_{\epsilon} L^2)^{B_0}_d(\Lambda_1\times \Lambda_2)= c_{0,w}(\Lambda_1;\ell^2(\Lambda_2))$ if $2 \leq p < \infty$,
\end{itemize}
topologically.
\end{corollary}

\section{Gabor frame characterisations of modulation spaces defined via TMIB and DTMIB}\label{sect-Gabor}
\subsection{The short-time Fourier transform and Gabor frames on $\mathcal{S'}(\R^n)$} \label{sect-STFT} We start with a brief discussion of the short-time Fourier transform (STFT) and Gabor frames on $L^2(\R^n)$; we refer to the book \cite{Grochenig} for more information. Recall that for $z= (x,\xi) \in \R^{2n}$ we write $\pi(z) = M_\xi T_x$. The STFT of $f \in L^2(\R^n)$ with respect to a window $\psi \in L^2(\R^n)$ is defined as
$$
V_\psi f(x,\xi) :=  (f, \pi(x,\xi)\psi)_{L^2} = \int_{\R^n} f(t) \overline{\psi(t-x)}e^{-2\pi i \xi \cdot t} \dt.
$$
Then, $V_\psi f \in L^2(\R^{2n})\cap C(\R^{2n})$ and the following orthogonality relation holds
\begin{equation}
\label{ortho}
(V_\psi f, V_{\gamma} \varphi)_{L^2} =  (f, \varphi)_{L^2}(\gamma,\psi)_{L^2},
\end{equation}
where also $\varphi, \gamma \in L^2(\R^n)$. Furthermore, it holds that
\begin{equation}
V_\psi ( \pi(x,\xi) f) = T^\sigma_{(x,\xi)} V_\psi f.
\label{shift-eq-1}
\end{equation}
Let $\psi, \gamma \in  L^2(\R^n)$ be such that $(\gamma,\psi)_{L^2} \neq 0$. The equations \eqref{ortho} and \eqref{shift-eq-1} imply the reproducing formula
\begin{equation}
V_{\varphi} f =  \frac{1}{(\gamma,\psi)_{L^2} }  V_{\psi} f \# V_{\varphi} \gamma,
\label{reproducing-1}
\end{equation}
where $f,\varphi \in L^2(\R^n)$.

Next, we discuss Gabor frames. Fix a lattice $\Lambda$  in $\R^{2n}$. Let $\psi \in L^2(\R^n)$ and suppose that the \emph{analysis operator}
$$
C_\psi:  L^2(\R^n) \rightarrow \ell^2(\Lambda), \, f \mapsto (V_\psi f(\lambda))_{\lambda \in \Lambda},
$$
is continuous; this is e.g.\  the case if $\psi \in W(L^\infty,L^1)$ \cite[Corollary 6.2.3]{Grochenig}. The adjoint operator of $C_\psi$, called the \emph{synthesis operator}, is given by
$$
D_\psi: \ell^2(\Lambda) \rightarrow L^2(\R^n), \, c \mapsto \sum_{\lambda \in \Lambda} c_\lambda \pi(\lambda) \psi,
$$
and the series $\sum_{\lambda \in \Lambda} c_\lambda \pi(\lambda)\psi$ converges unconditionally in $L^2(\R^n)$. Let $\psi, \gamma \in  L^2(\R^n)$ be windows such that   $C_\psi$ and  $C_\gamma$ are continuous. We define
$$
S_{\psi,\gamma} := D_\gamma \circ C_\psi: L^2(\R^n) \rightarrow L^2(\R^n)
$$
and call $(\psi,\gamma)$ a \emph{pair of dual windows} on $\Lambda$ if $S_{\psi,\gamma} = \operatorname{id}_{L^2(\R^n)}$. In such a case, also $S_{\gamma, \psi} = \operatorname{id}_{L^2(\R^n)}$ and thus
$$
f = \sum_{\lambda \in \Lambda} V_\psi f(\lambda) \pi(\lambda) \gamma = \sum_{\lambda \in \Lambda} V_\gamma f(\lambda) \pi(\lambda) \psi, \qquad f \in L^2(\R^n),
$$
where both series converge unconditionally in $L^2(\R^n)$.

Given a window $\psi \in L^2(\R^n)$, the set of time-frequency shifts
$$
\mathcal{G}(\Lambda, \psi) := \{ \pi(\lambda) \psi \, | \, \lambda \in \Lambda \}
$$
 is called a \emph{Gabor frame} if there are $A,B > 0$ such that
$$
A \| f \|_{L^2} \leq \| (V_\psi f (\lambda))_{\lambda \in \Lambda} \|_{\ell^2(\Lambda)}  \leq B \| f \|_{L^2}, \qquad f \in L^2(\R^n).
$$
Then, $S = S_{\psi,\psi}$ is a bounded positive invertible linear operator on $L^2(\R^n)$. Set $\gamma^\circ = S^{-1} \psi \in L^2(\R^n)$. Since $S$ and $\pi$ commute on $\Lambda$, $(\psi, \gamma^\circ)$ is a pair of dual windows on $\Lambda$. We call $\gamma^\circ$ the \emph{canonical dual window} on $\Lambda$ of $\psi$.

We now discuss the STFT and Gabor frames on  $\mathcal{S}'(\R^n)$ (cf.\ \cite[Section 11.2]{Grochenig} and  \cite{K-P-S-V}). Let $\psi \in \mathcal{S}(\R^n)$. Then, the mapping
$V_\psi : \mathcal{S}(\R^n) \rightarrow \mathcal{S}(\R^{2n})$ is continuous. The STFT of $f \in \mathcal{S}'(\R^n)$ with respect to $\psi$ is defined as
\begin{equation}
V_\psi f(x,\xi) := \langle f,\pi(x,-\xi)  \overline{\psi} \rangle, \qquad (x,\xi) \in \R^{2n}.
\label{def-STFT-dual}
\end{equation}
Then, $V_\psi f \in C(\R^{2n})$ and $\| V_\psi f\|_{L^{\infty}_{(1+| \, \cdot \, |)^{-N} }}<\infty$ for some $N \in \N$. If $A \subset  \mathcal{S}'(\R^n)$  is bounded, then the previous estimate holds uniformly for $f \in A$. Since $\mathcal{S}'(\R^n)$ is bornological, this implies that the mapping $V_\psi : \mathcal{S}'(\R^n) \rightarrow \mathcal{S}'(\R^{2n})$ is continuous. Let $\psi, \gamma \in  \mathcal{S}(\R^n)$ be such that $(\gamma,\psi)_{L^2} \neq 0$. As $L^2(\R^n)$ is dense in $\mathcal{S}'(\R^n)$, \eqref{ortho} implies that
\begin{equation}
\langle f, \varphi \rangle = \frac{1}{(\gamma,\psi)_{L^2}}\iint_{\R^{2n}} V_\psi f(x,\xi) V_{\overline{\gamma}} \varphi(x,-\xi) \dx\dxi, \qquad \varphi \in \mathcal{S}(\R^n),
\label{inversion-dual}
\end{equation}
whereas \eqref{reproducing-1} yields that
\begin{equation}
V_{\varphi} f =  \frac{1}{(\gamma,\psi)_{L^2} }  V_{\psi} f \# V_{\varphi}\gamma, \qquad  f \in \mathcal{S}'(\R^n),\, \varphi\in\SSS(\R^n).
\label{reproducing-2}
\end{equation}
Clearly, \eqref{shift-eq-1} remains true for $f\in\SSS'(\R^n)$ and $\psi\in\SSS(\R^n)$.\\
\indent Finally, we discuss Gabor frames on $\mathcal{S}'(\R^n)$. Let $\psi \in  \mathcal{S}(\R^n)$. The mappings
$$
C_\psi:  \mathcal{S}'(\R^n) \rightarrow \mathcal{S}'_d(\Lambda) , \, f \mapsto (V_\psi f(\lambda))_{\lambda \in \Lambda},
$$
and
$$
D_\psi: \mathcal{S}'_d(\Lambda) \rightarrow \mathcal{S}'(\R^n), \, c \mapsto \sum_{\lambda \in \Lambda} c_\lambda \pi(\lambda) \psi,
$$
are well-defined and continuous, and the series $\sum_{\lambda \in \Lambda} c_\lambda \pi(\lambda)\psi$ is absolutely summable in $\mathcal{S}'(\R^n)$. Let $\psi, \gamma \in  \mathcal{S}(\R^n)$  be such that $(\psi, \gamma)$ is a pair of dual windows on $\Lambda$. Then,
$$
f = \sum_{\lambda \in \Lambda} V_\psi f(\lambda) \pi(\lambda) \gamma = \sum_{\lambda \in \Lambda} V_\gamma f(\lambda) \pi(\lambda) \psi, \qquad f \in \mathcal{S}'(\R^n),
$$
where both series are absolutely summable in $\mathcal{S}'(\R^n)$.

\subsection{Continuity of the Gabor frame operators on modulation spaces associated to TMIB and DTMIB} Fix a TMIB or a DTMIB $F$ on $\R^{2n}$.
We start by defining the modulation space associated to $F$ \cite{D-P-P-V}.
\begin{definition}
Let $\psi \in \mathcal{S}(\R^n)\backslash \{0\}$. We define the modulation space associated to $F$ as
$$
\mathcal{M}^F := \{ f \in \mathcal{S}'(\R^n) \, | \, V_\psi f \in F \}
$$
and endow it with the norm $\|f \|_{\mathcal{M}^F} := \| V_\psi f \|_F$.
\end{definition}
We sometimes employ the alternative notation $\mathcal{M}[F]$ for $\mathcal{M}^F$. The space $\mathcal{M}^F$ is a Banach space whose definition is independent of the window $\psi\in\SSS(\R^n)\backslash\{0\}$ and different non-zero windows induce equivalent norms on $\mathcal{M}^F$ \cite[Corollary 4.5 and  Corollary 4.6]{D-P-P-V}. Furthermore, if $F$ is a TMIB, then $\mathcal{M}^F$ is a TMIB \cite[Theorem 4.8$(i)$]{D-P-P-V}. We define
$$
\check{F}_2 := \{ f \in \mathcal{S}'(\R^{2n}) \, | \, \check{f}_2(x,\xi) := f(x,-\xi) \in F\}
$$
and endow it with the norm $\| f \|_{\check{F}_2} := \| \check{f}_2\|_F$. It is clear that $\check{F}_2$ is again a TMIB (DTMIB). The following duality result holds.
\begin{proposition} \label{modulation-duality} \cite[Theorem 4.8$(iii)$]{D-P-P-V}
Suppose that $F$ is a TMIB. Then, $\mathcal{M}^{F'} = (\mathcal{M}^{\check{F}_2})'$. Moreover, for $\psi, \gamma \in \mathcal{S}(\R^n)$ with  $(\gamma, \psi)_{L^2} \neq 0$, it holds that (cf.\ \eqref{inversion-dual})
$$
\langle f, g \rangle = \frac{1}{(\gamma, \psi)_{L^2}} \langle V_\psi f(x,\xi), V_{\overline{\gamma}} g(x,-\xi) \rangle, \qquad f \in \mathcal{M}^{F'},\, g \in  \mathcal{M}^{\check{F}_2}.
$$
\end{proposition}
\noindent Consequently, $\mathcal{M}^F$ is a DTMIB if $F$ is so.

\begin{remark}
The identity \eqref{shift-eq-1} implies that
$$
\| \pi(x,\xi) \|_{\mathcal{L}(\mathcal{M}^F)} \leq \rho^{B_0}_F(x,\xi), \qquad (x,\xi) \in \R^{2n}.
$$
Hence, \cite[Theorem 12.1.9]{Grochenig}  gives the  continuous inclusion
\begin{equation}
\label{inclusion-rho}
\mathcal{M}^{1,1}_{\rho^{B_0}_F} \subseteq \mathcal{M}^F,
\end{equation}
which improves \cite[Corollary 4.11]{D-P-P-V}.
\end{remark}

For the main result of this article we need to enlarge the class of windows for the STFT of the elements of $\mathcal{M}^F$ in such a way that its range consists of continuous functions on $\R^{2n}$. Given a Banach space $X \subset \mathcal{S}'(\R^n)$, we define the Banach space $\overline{X} :=\{f\in\SSS'(\R^n)\,|\, \overline{f}\in X\}$ with norm $\|f\|_{\overline{X}}:=\|\overline{f}\|_X$. Assume that $F$ is a TMIB. For $f \in \mathcal{M}^F$ and $\psi \in \overline{\mathcal{M}[(F')\check{}_2]}$ we define
$$
V_\psi f(x,\xi) := {}_{\mathcal{M}^F}\langle f,\pi(x,-\xi)  \overline{\psi} \rangle_{\mathcal{M}^{(F')\check{}_2}}.
$$
Similarly, for $f \in \mathcal{M}^{F'}$ and $\psi \in \overline{\mathcal{M}[\check{F}_2]}$ we define
$$
V_\psi f(x,\xi) := {}_{\mathcal{M}^{F'}}\langle f,\pi(x,-\xi)  \overline{\psi} \rangle_{\mathcal{M}^{\check{F}_2}}.
$$
Obviously, these definitions coincide with the one given in \eqref{def-STFT-dual} if $\psi \in \mathcal{S}(\R^n)$. Since $(T^{B_0}_{(x,-\xi)} G)\check{}_2 = T^{-B_0}_{(x,\xi)} \check{G}_2$ for all $G \in \mathcal{S}'(\R^{2n})$, Proposition \ref{modulation-duality} together with \eqref{shift-eq-1} imply that the sesquilinear mappings
$$
\mathcal{M}^F\times \overline{\mathcal{M}[(F')\check{}_2]} \rightarrow C_{1/\check{\rho}^{B_0}_F}(\R^{2n}),\, (f,\psi)\mapsto V_{\psi} f
$$
and
$$
 \mathcal{M}^{F'}\times \overline{\mathcal{M}[\check{F}_2]} \rightarrow C_{1/\check{\rho}^{B_0}_{F'}}(\R^{2n}),\, (f,\psi)\mapsto V_{\psi} f
$$
are well-defined and continuous. Now suppose again that $F$ is either a TMIB or a DTMIB. Since
\begin{gather}
\overline{V_\psi f} = (V_{\overline{\psi}} \overline{f})\check{}_2,\qquad f \in \mathcal{S}'(\R^n),\,\psi \in \mathcal{S}(\R^n),\label{short-timr-conj-sec-variab}\\
\overline{W(\mathcal{F}L^1_{\check{\widetilde{\nu}}_F}, L^1_{\check{\omega}_F})} = W(\mathcal{F}L^1_{\widetilde{\nu}_F}, L^1_{\check{\omega}_F}),\label{conjugate-amalgam}
\end{gather}
Corollary \ref{inclusion-amalgam-E} implies that $\mathcal{M}[W(\mathcal{F}L^1_{\check{\widetilde{\nu}}_F}, L^1_{\check{\omega}_F})] \subset  \overline{\mathcal{M}[(F')\check{}_2]}$ continuously  if $F$ is a TMIB and $\mathcal{M}[W(\mathcal{F}L^1_{\check{\widetilde{\nu}}_F}, L^1_{\check{\omega}_F})] \subset  \overline{\mathcal{M}[(F_0)\check{}_2]}$ continuously if $F$ is a DTMIB with  $F = F'_0$, where $F_0$ is a TMIB. Hence, the sesquilinear mapping
\begin{equation}\label{continu-map-extens-f}
\mathcal{M}^F\times \mathcal{M}[W(\mathcal{F}L^1_{\check{\widetilde{\nu}}_F}, L^1_{\check{\omega}_F})] \rightarrow C_{1/\check{\rho}^{B_0}_F}(\R^{2n}),\, (f,\psi)\mapsto V_{\psi} f
\end{equation}
is well-defined and continuous.

\begin{remark} Although we will not need this, we would like to point out that it is  also possible to enlarge the class of windows for the STFT of the elements in $\mathcal{M}^F$ in such a way that its range is in $F$:
\begin{proposition}
The sesquilinear mapping $\mathcal{M}^F\times \SSS(\R^n)\rightarrow F$, $(f,\psi)\mapsto V_{\psi} f$, uniquely extends to a continuous sesquilinear mapping $\mathcal{M}^F\times M^{1,1}_{\check{\rho}^{B^t_0}_F}\rightarrow F$, $(f,\psi)\mapsto V_{\psi} f$.
\end{proposition}

\begin{proof}
For all $G \in \mathcal{S}'(\R^{2n})$ and $\Phi \in  \mathcal{S}(\R^{2n})$, it holds that
\begin{equation}
G \# \Phi = \iint_{\R^{2n}} \Phi(x,\xi) T^{B^t_0}_{(x,\xi)} G \dx \dxi,
\end{equation}
where the integral should be interpreted as an $\mathcal{S}'(\R^n)$-valued Pettis integral with respect to the weak-$\ast$ topology on $\mathcal{S}'(\R^n)$. If $G \in F$, then the above integral exists as an $F$-valued Bochner integral if $F$ is a TMIB and as an $F$-valued Pettis integral if $F$ is a DTMIB. Consequently, $G \# \Phi  \in F$ and
\begin{equation}
\| G \# \Phi \|_F  \leq \|G\|_F \|\Phi\|_{L^1_{\rho^{B^t_0}_F }}, \qquad G \in F,\, \Phi  \in  \mathcal{S}(\R^{2n}).
\label{bound-twisted-F}
\end{equation}
Now fix $\gamma \in \mathcal{S}(\R^n)$ with $\|\gamma\|_{L^2} = 1$. Let $f \in \mathcal{M}^F$ and $\psi \in \mathcal{S}(\R^n)$ be arbitrary. Note that $V_{\psi}\gamma= (\theta_{-B_0}(\overline{V_{\gamma}\psi}))\check{}$. Hence, the reproducing formula \eqref{reproducing-2} and \eqref{bound-twisted-F} yield that
$$
\| V_\psi f \|_F = \| V_\gamma f \# V_\psi \gamma \|_F \leq \|  V_\gamma f \|_{F}  \|V_\psi \gamma \|_{L^1_{\rho^{B^t_0}_F }} = \|  V_\gamma f \|_{F}   \| V_\gamma \psi \|_{L^1_{\check{\rho}^{B^t_0}_F }},
$$
whence the result follows from the density of $\SSS(\R^n)$ in $M^{1,1}_{\check{\rho}^{B^t_0}_F}$.
\end{proof}

\end{remark}
Fix a lattice $\Lambda$ in $\R^{2n}$ and a bounded open neighbourhood $U$ of the origin in $\R^{2n}$ such that the family of sets $\{ \lambda + U \, | \, \lambda \in \Lambda \}$ is pairwise disjoint. We are ready to establish the continuity of the analysis and synthesis operators on $\mathcal{M}^F$. Recall that $F^\sigma_d(\Lambda) = F^{B_0}_d(\Lambda)$.

\begin{theorem}\label{analysis-coorbit} ${}$
\begin{itemize}
\item[$(i)$] Let $\psi \in  \mathcal{M}[W(\mathcal{F}L^1_{\check{\widetilde{\nu}}_F}, L^1_{\check{\omega}_F})]$. The mapping $C_{\psi} :  \mathcal{M}^F \rightarrow F^\sigma_{d}(\Lambda)$ is well-defined and continuous.
\item[$(ii)$] Let $\psi \in  \mathcal{M}[W(\mathcal{F}L^1_{\widetilde{\nu}_F}, L^1_{\omega_F})]$. For each $c\in F^{\sigma}_d(\Lambda)$ the series $\sum_{\lambda \in \Lambda} c_\lambda \pi(\lambda)\psi$ is C\'esaro summable in $\mathcal{M}^F$ if $F$ is a TMIB and C\'esaro summable with respect to the weak-$\ast$ topology on $\mathcal{M}^F$ if $F$ is a DTMIB (cf.\ Proposition \ref{modulation-duality}). Furthermore, the mapping $D_{\psi} :  F^\sigma_{d}(\Lambda) \rightarrow \mathcal{M}^F$ is well-defined and continuous.
\end{itemize}
\end{theorem}
\begin{proof} $(i)$ Let $f \in \mathcal{M}^F$ be arbitrary. As $V_{\psi} f$ is continuous, we can evaluate it  at $\lambda\in\Lambda$. Pick $\gamma \in \mathcal{S}(\R^n)$ such that $\|\gamma\|_{L^2}=1$. Note that, by \eqref{conjugate-amalgam},
$$
V_{\psi}\gamma= (\theta_{-B_0}(\overline{V_{\gamma}\psi}))\check{}  \in (\theta_{-B_0}W(\mathcal{F}L^1_{\widetilde{\nu}_F}, L^1_{\check{\omega}_F}))\check{}.
$$
Since $\mathcal{S}(\R^n)$ is dense in $\mathcal{M}[W(\mathcal{F}L^1_{\check{\widetilde{\nu}}_F}, L^1_{\check{\omega}_F})]$, the reproducing formula \eqref{reproducing-2} and the continuity of the mappings in \eqref{amalgam-twisted-conv} and \eqref{continu-map-extens-f} imply that $V_\psi f = V_\gamma f \# V_\psi \gamma$. Hence, the result follows from  Proposition \ref{sampling}.\\
$(ii)$ In view of \eqref{shift-eq-1}, this is a consequence of Proposition \ref{lemma-synthesis} and Corollary \ref{equ-for-bilinear-mapping-sum} (and Proposition \ref{modulation-duality} if $F$ is a DTMIB).
\end{proof}

\begin{corollary}\label{cor-for-rep11}
Let $\psi \in \mathcal{M}[W(\mathcal{F}L^1_{\check{\widetilde{\nu}}_F}, L^1_{\check{\omega}_F})]\cap L^2$ and $\gamma \in\mathcal{M}[W(\mathcal{F}L^1_{\widetilde{\nu}_F}, L^1_{\omega_F})]\cap L^2$ be such that $(\psi,\gamma)$ is a pair of dual windows on $\Lambda$. Then,
\begin{equation}\label{equ-for-tms}
f = \sum_{\lambda \in \Lambda} V_{\psi} f(\lambda) \pi(\lambda) \gamma, \qquad f \in \mathcal{M}^F,
\end{equation}
where the series is C\'esaro summable in $\mathcal{M}^F$ if $F$ is a TMIB and C\'esaro summable with respect to the weak-$\ast$ topology on $\mathcal{M}^F$ if $F$ is a DTMIB. Furthermore, there are $A,B > 0$ such that
$$
A\| f \|_{\mathcal{M}^F} \leq \| (V_{\psi} f(\lambda))_{\lambda \in \Lambda} \|_{F^\sigma_{d}(\Lambda)}  \leq B \| f \|_{\mathcal{M}^F}, \qquad f \in \mathcal{M}^F.
$$
\end{corollary}

\begin{proof} Note that $D_{\gamma}\circ C_{\psi}$ restricts to the identity on $\SSS(\R^n)$. Hence, if $F$ is a TMIB, the result follows from the density of $\SSS(\R^n)$ in $\mathcal{M}^F$ and Theorem \ref{analysis-coorbit}. Assume now that $F$ is a DTMIB. Theorem \ref{analysis-coorbit} and \eqref{short-timr-conj-sec-variab} imply that for all $\chi\in\SSS(\R^n)$ and $f\in\mathcal{M}^F$
$$
\langle D_{\gamma}C_{\psi}(f),\chi\rangle =\lim_{N\rightarrow\infty} \overline{\left\langle \overline{f}, \sum_{|m_j|<N}\left(1-\frac{|m_1|}{N}\right)\cdots \left(1-\frac{|m_{2n}|}{N}\right) V_{\gamma}\overline{\chi}(A_{\Lambda}m)\pi(A_{\Lambda} m)\psi\right\rangle},
$$
whence the claim follows from the part of the corollary about TMIB and Theorem \ref{analysis-coorbit}.
\end{proof}
We now give two remarks about the window classes employed in Theorem \ref{analysis-coorbit} and Corollary \ref{cor-for-rep11}.
\begin{remark}
Let $\omega$ and $\nu$ be submultiplicative polynomially bounded weight functions on $\R^{2n}$ and set $X=W(\mathcal{F}L^1_{\nu},L^1_{\omega})$. Then, $\omega_X(x,\xi) \leq C\omega(x,\xi)$ and $\nu_X(x,\xi) \leq C\nu(x,\xi)$. Hence, \eqref{twisted-wf} and \eqref{inclusion-rho} gives the inclusions
\begin{equation}
\label{mod-am-1}
M^{1,1}_{\check{\sigma}_F}\subseteq \mathcal{M}[W(\mathcal{F}L^1_{\check{\widetilde{\nu}}_F}, L^1_{\check{\omega}_F})]\quad \mbox{and}\quad M^{1,1}_{\sigma_F}\subseteq \mathcal{M}[W(\mathcal{F}L^1_{\widetilde{\nu}_F}, L^1_{\omega_F})],
\end{equation}
where $\sigma_F(x,\xi) = \omega_F(x,\xi)\tilde{\nu}_F(0,x)$. If $\nu_F(0, \, \cdot \,) = 1$, the above inequality and the inclusion $W(\mathcal{F}L^1, L^1_{\omega_F}) \subseteq  W(L^\infty, L^1_{\omega_F}) \subset  L^1_{\omega_F}$ imply that
\begin{equation}
\label{mod-am-2}
M^{1,1}_{\check{\omega}_F} = \mathcal{M}[W(\mathcal{F}L^1, L^1_{\check{\omega}_F})]\quad \mbox{and} \quad M^{1,1}_{\omega_F} = \mathcal{M}[W(\mathcal{F}L^1, L^1_{\omega_F})].
\end{equation}
By \eqref{mod-am-1}, we can take $\psi \in M^{1,1}_{\check{\sigma}_F}$ in Theorem \ref{analysis-coorbit}$(i)$ and $\psi \in M^{1,1}_{\sigma_F}$ in Theorem \ref{analysis-coorbit}$(ii)$; a similar statement holds for Corollary \ref{cor-for-rep11}. As mentioned in the introduction, if $F$ is solid, Theorem \ref{analysis-coorbit} and Corollary \ref{cor-for-rep11}  are known to hold true for the window class $M^{1,1}_{\max\{\omega_F,\check{\omega}_F\}}$ \cite{F-G1, Grochenig1991, Grochenig}. The equalities in \eqref{mod-am-2} imply that this remains valid for the larger class of TMIB and DTMIB $F$ for which $\nu_F(0, \, \cdot \,) = 1$; e.g $F =  E_1 \widehat{\otimes}_\tau E_{2}$,  $\tau = \pi$ or $\epsilon$, where $E_1$ is a TMIB on $\R^n$ and $E_2$ is a solid TMIB on $\R^n$, satisfy $\nu_F(0, \, \cdot \,) = 1$.
\end{remark}

\begin{remark}\label{windows-for-the special-latticeon-s}
For each $\varphi \in W(L^{\infty},L^1) \backslash \{ 0\}$, the system $\mathcal{G}(\varphi, a\Z^n \times b \Z^n)$ is a Gabor frame for $a,b > 0$ small enough \cite[Theorem 6.5.1]{Grochenig}. If $\varphi(x) = 2^{n/4} e^{-\pi x \cdot x}$ is the Gaussian, $\mathcal{G}(\varphi, a\Z^n \times b \Z^n)$ is a Gabor frame if and only if $ab < 1$ (cf.\ \cite[Theorem 7.5.3]{Grochenig}). If $\varphi \in \mathcal{S}(\R^n)$ and $\mathcal{G}(\varphi, a \Z^n \times b\Z^n)$ is a Gabor frame, then the canonical dual window $\gamma^0 = S^{-1}\varphi$ on $a \Z^n \times b\Z^n$ also belongs to $\mathcal{S}(\R^n)$ \cite[Corollary 13.5.4]{Grochenig} (see \cite{G-L} for a more refined version of this result).
\end{remark}

We end this article by giving two applications of Corollary \ref{cor-for-rep11}. The next result and various related statements were recently shown in \cite{f-p-p} via different methods.
\begin{corollary}\label{cor-mod-tensor}
Let $w_1$ and $w_2$ be two polynomially bounded weight functions on $\R^n$ and let $1 \leq p_1, p_2 < \infty$. Then,
\begin{itemize}
\item[$(i)$] $\mathcal{M}[L^{p_1}_{w_1}\widehat{\otimes}_{\pi} \mathcal{F}L^{p_2}_{w_2}] = \mathcal{M}[L^1_{w_1w_2}( \mathcal{F}L^{p_2})] =  W(L^{p_2},L^1_{w_1w_2})$  if $p^{-1}_1+p^{-1}_2 \geq  1$,
\item[$(ii)$] $\mathcal{M}[L^{p_1}_{w_1}\widehat{\otimes}_{\epsilon} \mathcal{F}L^{p_2}_{w_2}] =   \mathcal{M}[C_{0,w_1w_2}( \mathcal{F}L^{p_2})] = W(L^{p_2},C_{0,w_1w_2})$ if $p^{-1}_1+p^{-1}_2 \leq  1$,
\end{itemize}
topologically.
\end{corollary}

\begin{proof}
In view of Corollary \ref{cor-for-rep11} (and Remark \ref{windows-for-the special-latticeon-s}), the topological identities
\begin{gather*}
\mathcal{M}[L^{p_1}_{w_1}\widehat{\otimes}_{\pi} \mathcal{F}L^{p_2}_{w_2}] =  \mathcal{M}[L^1_{w_1w_2}( \mathcal{F}L^{p_2})]  , \qquad  p^{-1}_1+p^{-1}_2 \geq  1, \\
\mathcal{M}[L^{p_1}_{w_1}\widehat{\otimes}_{\epsilon} \mathcal{F}L^{p_2}_{w_2}] = \mathcal{M}[C_{0,w_1w_2}( \mathcal{F}L^{p_2})] , \qquad  p^{-1}_1+p^{-1}_2 \leq  1,
\end{gather*}
follow from Proposition \ref{mixed-norm-TMIB} and Proposition \ref{seq-spa-for-the-pi-prod}. The proof of the other two identities is straightforward and we omit them.
\end{proof}
Corollary \ref{cor-for-rep11} (and Remark \ref{windows-for-the special-latticeon-s}) also imply that modulation spaces defined via TMIB satisfy the sequential approximation property \cite[Chapter 43]{kothe2}; we refer to \cite{drw} for more information on approximation properties for the classical modulation spaces $M^{p,q}_w$, $1\leq p,q<\infty$.
\begin{corollary}
Let $F$ be a TMIB on $\R^{2n}$. Then, $\mathcal{M}^F$ satisfies the sequential approximation property, that is, there exists a sequence of finite rank operators $(P_n)_{n\in\N}\subset (\mathcal{M}^F)'\otimes \mathcal{M}^F$  which converges to $\operatorname{id}_{\mathcal{M}^F}$ in $\mathcal{L}_p(\mathcal{M}^F)$, where $p$ stands for the topology of uniform convergence on precompact sets.
\end{corollary}

\end{document}